 \documentclass{amsart}
\usepackage{myAMSpackages}
\usepackage{MyAMSEnvironments}

\title{Isometries of complemented sub-Riemannian manifolds}

\author{Robert K. Hladky}
\address{NDSU Mathematics Dept \#2750, PO Box 6050, Fargo ND 58108-6050,USA }

\email{robert.hladky@ndsu.edu}

\newcommand{\aip}[3]{\ensuremath{\langle \, {#1} \, , \, {#2} \, \rangle_{#3}}\xspace}
\newcommand{\pd}[2]{\ensuremath{\frac{\partial {#1}}{\partial {#2}}}\xspace}

\DeclareFontFamily{OT1}{rsfs}{}
\DeclareFontShape{OT1}{rsfs}{m}{n}{ <-7> rsfs5 <7-10> rsfs7 <10-> rsfs10}{} 
\DeclareMathAlphabet{\mathscr}{OT1}{rsfs}{m}{n}

\newcommand{\tor}{\text{Tor}}
\newcommand{\Tor}[1][2]{\text{TOR}_{#1}}
\newcommand{\src}{\ensuremath{\text{Rc}^s}\xspace}
\newcommand{\srm}{\ensuremath{\text{Rm}^s}\xspace}

\newcommand{\lp}{\triangle_H}
\newcommand{\mB}[1][K]{\ensuremath{\mathcal{B}_{#1}}\xspace}

\keywords{Carnot-Carath\'eodory geometry, isometry, Killing field, Ricci curvature, Bochner Formula}

\begin{document}

\begin{abstract}
We show that the group of smooth isometries of a complemented sub-Riemannian manifold form a Lie group and establish dimension estimates based on the torsion of the canonical connection. We explore the interaction of curvature and the structure of isometries and Killing fields and derive a Bochner formula for Killing fields. Sub-Riemannian generalizations of classical results of Bochner and Berger are established. We also apply our theory to common sub-categories of complemented sub-Riemannian geometries and show how to compute the isometry groups for several examples, including $SO(n)$, $SL(n)$ and the rototranslation group. 
\end{abstract}

\maketitle

\section{Introduction}\label{S:ID}

The essential idea of Klein's Erlangen program is to study geometric objects via their groups of self-isometries. While general Riemannian geometry falls outside the boundaries of this program, there are deep and rich interactions between the isometry groups of a Riemannian manifolds and tensorial invariants such as curvature. 

There is a natural notion of isometry in sub-Riemannian geometry. But, in part due to the problematic nature of the exponential map on sub-Riemannian manifolds, there has been little success in studying the group of isometries using analytic tensorial methods.

In a recent paper \cite{Hladky4}, the author introduced a canonical connection for complemented sub-Riemannian manifolds and studied some of the basic properties of its curvature and torsion. Using the language of Killing fields and one-parameter subgroups, this connection can be used to study the isometries  that preserve not only the sub-Riemannian distance but also the complement. Many natural examples of sub-Riemannian geometries come equipped with an inbuilt or natural complement. The isometries of interest in these examples are exactly those that preserve this additional structure.

In this paper, we establish the basic definitions and properties for isometries and Killing fields on complemented sub-Riemannian manifolds. In Section \ref{Iso}, we show that isometries can be determined by information at a single point, establish that $\text{Iso}(M)$ is a finite dimensional Lie group and find crude dimension estimates. In Section \ref{Tor}, we study the considerable effect of torsion on the group of isometries and establish methods to greatly refine the dimension estimate. In Section \ref{Neg}, we first  study the pointwise relationships between Killing fields and curvature.  Then we consider compact manifolds and Bochner formulas for Killing fields. We extend classical results of Bochner and Berger relating Killing fields to Ricci curvature into the  sub-Riemannian category. Finally in Section \ref{Ex} we show how to apply our results to a variety of examples. We show that the Lie group $SO(n)$ with a natural left-invariant, complemented sub-Riemannian structure provides an example of a compact manifold with  isometry group of maximal size.

\section{Basic Definitions}\label{S:BD}

\begin{defn}\label{D:sub-R}
A {\em sub-Riemannian manifold} is a smooth manifold $M$ together with a smooth, bracket generating  subbundle $HM$ of the tangent bundle and a smooth inner product $\aip{}{}{}$ on $HM$.
\end{defn}

It is well known that the sub-Riemannian metric defines an equivalent topology to that of the underlying manifold $M$. We shall always assume that $M$ is complete with respect to the metric induced by the sub-Riemannian distance. From \cite{Strichartz}  Theorem 7.4 , this is equivalent to the completeness with respect to any particular extension of $\aip{}{}{}$ to a full Riemannian metric.

To allow for precise statements concerning the bracket structure of $M$, we define $H^0=HM$ and then inductively define $H^{j+1}_p$ to be the linear hull of $H^j_p$ and all vectors of the form $[X,Z]_p$ with $X \in \Gamma^\infty(H^0)$ and $Z$ a smooth vector field that everywhere is in $H^j$.  The sub-Riemannian manifold $M$ has step size $r$ at $p$ if $H^{r-1}_p = T_pM$ and global step size $r$ if $H^{r-1} =TM$. We say that the sub-Riemannian manifold $M$ is {\em regular} if each $H^j$ is a smooth distribution of constant rank. 

\begin{defn}\label{D:sRC}
A {\em sub-Riemannian manifold with complement (sRC-manifold)} is a sub-Riemannian manifold together with a smooth bundle $VM$ such that $TM = HM \oplus VM$.

An { \em $r$-grading} on an sRC-manifold is a decomposition 
\[ VM = V^1 \oplus \dots \oplus V^r\]
where each $V^i$ is a smooth bundle and we have the bracket conditions 
\begin{align*}
[\Gamma^\infty(HM),\Gamma^\infty(HM)] &\subseteq \Gamma^\infty(HM\oplus V^1)\\ [\Gamma^\infty(HM),\Gamma^\infty(V^i)] &\subseteq \Gamma^\infty(HM \oplus V^1\oplus \dots \oplus V^{i+1}).
\end{align*} 
A {\em regular} sRC-manifold is a regular sub-Riemannian manifold of step size $r+1$ together with an $r$-grading such that
\[ H^j = H^{j-1} \oplus V^{j}\] 
for all $j=1,\dots, r$.
\end{defn}
 
 \begin{rem}\label{R:basic}
 There is only one $1$-grading on any sRC-manifold. We shall refer to this as the basic grading. An sRC-manifold with the basic grading is always regular. 
 
 To simplify notation, we shall also set $V^0=H$ and use $\pi^j$, $j=0,\dots,r$ to represent the projection maps $TM \to V^j$. However, when we are working with the basic grading, we shall continue to use $HM$ and $VM$. The projections of a vector field $A$ will often be denoted $A_H$ and $A_V$.
 \end{rem}

\begin{defn}\label{D:HIsometry} A {\em weak $H$-isometry} between sRC-manifolds is a smooth diffeomorphism $\varphi \colon M \to N$ such  that
\[ \varphi_* HM = HN , \qquad \varphi^* \aip{}{}{N} = \aip{}{}{M}\]
An {\em $H$-isometry} is a weak $H$-isometry such that
\[  \varphi_* VM =VN\]
If both $M$ and $N$ are regular sRC-manifolds of step size $r+1$, a {\em regular $H$-isometry} is an $H$-isometry such that $\varphi_* V_M^j = V^j_N$ for all $j$.

The group of $H$-isometries from an sRC-manifold $M$ to itself will be denoted $\text{Iso}(M)$ with the weak and regular groups denoted by $\text{Iso}^*(M)$ and $\text{Iso}^R(M)$ respectively.

\end{defn}

\begin{rem}
It would certainly be possible to define $H$-isometries under much weaker regularity assumptions than smoothness, although some regularity is necessary for the preservation of the horizontal bundle to make sense. In this paper, however, we shall restrict our attention to the smooth category. 
\end{rem}

\begin{defn}\label{D:1P}
A {\em one-parameter subgroup of (weak/regular) $H$-isometries } on $M$ is a group homomorphism $\Phi$ from $\mathbb{R}^{}$ to the group of (weak/regular) $H$-isometries, such that for all $p\in M$, the map $t \mapsto \Phi(t,p)$ describes a smooth curve.  
\end{defn}

\begin{defn}\label{D:Killing}
A smooth vector field $K$ is a {\em weak  $H$-Killing field} if $ [K,HM]  \subset HM $
and for all sections $X,Y $ of $HM$
\[  K\aip{X}{Y}{} = \aip{[K,X]}{Y}{} + \aip{X}{[K,Y]}{}. \]
 
A vector field $K$ is an {\em $H$-Killing field} if it is a weak $H$-Killing field such that 
$[K,VM]\subset VM$. If $M$ is a regular sRC-manifold, then an $H$-Killing field is {\em regular} if $[K,V^j] \subset V^j$ for all $j$.
\end{defn}

When we need to emphasize the distinction from weak $H$-Killing fields, we shall occasionally refer to $H$-Killing fields as {\em strong} $H$-Killing fields.

\begin{lemma}\label{L:Killing}\hfill

\begin{enumerate}
\item The local flows of a (weak/regular) $H$-Killing field act by (weak/regular) $H$-isometries.
\item Every (weak/regular) $H$-Killing field is complete.
\item The map $\Phi \mapsto  \frac{d}{dt}_{| t=0} \Phi(t,p)$  defines a $1-1$ correspondence between one-parameter families of (weak/regular) $H$-isometries and (weak/regular) $H$-Killing fields.
\end{enumerate}
\end{lemma}

The proof requires only trivial modifications of the Riemannian version, see for example \cite{Petersen} p.188.

In the majority of interesting cases the horizontal bundle $HM$ is strongly non-integrable. This means that typically there will be no purely horizontal $H$-Killing fields.  However, $H$-Killing fields can have horizontal components. For example on Carnot groups, the infinitesimal generators of left translations are $H$-Killing fields.

\begin{exm}\label{X:Heisn}
The n-th Heisenberg group $\mathbb{H}^{n}$ is $\mathbb{R}^{2n+1} =\mathbb{R}^{n}_x \times \mathbb{R}^{n}_y \times \mathbb{R}^{}_t$ with $H\mathbb{H}^{n}$ spanned by the orthonormal basis $X_i = \pd{}{x^i} -\frac{y^i}{2} \pd{}{t}$, $Y_i = \pd{}{y^i} +\frac{x^i}{2} \pd{}{t}$, $i=1,\dots,n$. A complement is defined by choosing $T = \pd{}{t}$.  The group structure
\[ (x,y,t) \cdot (\tilde{x},\tilde{y},\tilde{t}) = \left(x+\tilde{x}, y+\tilde{y}, t+\tilde{t} + \frac{1}{2} (x^i \tilde{y}^i - \tilde{x}^i y^i ) \right) \]
makes $\mathbb{H}^{n}$ a Lie group for which each $X_i,Y_i,T$ is left invariant. Hence the group of left translations is a transitive subgroup of $\text{Iso}(\mathbb{H}^{})$.

 Identify $\mathfrak{u}(n)$ with the space of real $2n \times 2n$ matrices of the form $(A,B)= \begin{pmatrix} A & B \\ -B & A \end{pmatrix}$ with $A$ skew-symmetric and $B$ symmetric. For $(A,B) \in \mathfrak{u}(n)$, it is easy to verify that
\[\begin{split} K_{A,B} & = (A_{ij} x^j + B_{ij} y^j)  X_i + (A_{i j}y^j - B_{i j}x^j ) Y_i  \\
 & \qquad +  \left(  A_{ij} x^j y^i   + \frac{1}{2}   B_{ij} (x^ix^j +y^i y^j) \right) T \end{split} \]
 is an $H$-Killing field. It is well known that these form a basis for the space of $H$-Killing fields on $\mathbb{H}^{n}$ and indeed it will later follow from Corollary \ref{C:psidim} that the Lie algebra of smooth $H$-Killing fields on $\mathbb{H}^{n}$ is isomorphic to $\mathfrak{u}(n)$.

 For the Heisenberg groups, we can actually make a much stronger statement. Let
 \[ K = a^i X_i + b^i Y_i + cT\]
 be  any weak $H$-Killing field. Then it is easy to verify that 
 \[ X_i a^j +X_ja^i =0 = Y_i b^j +Y_j b^i, \qquad X_i b^j +Y_j a^i = 0\]
 and that
 \[ X_i c = -b^i, \qquad Y_i c = a^i.\]
 But then for any fixed $i$, $Tc = [X_i,Y_i]c = X_i a^i + Y_i b^i = 0 $ and therefore
 \[ Ta^i = TY_i c = Y_i Tc =0, \qquad Tb^i  = -TX_i c = -X_i Tc =0.\]
 From this it is easy to see that $[K,T]_H =0$ and so $K$ is actually a strong $H$-Killing field. Thus the Lie algebras of weak and strong $H$-Killing fields agree.

\end{exm}

To compute with $H$-Killing fields we shall use specialized connections associated to an sRC-manifold.  We begin by establishing some notation and terminology.

The following theorem is shown in \cite{Hladky4}.

\begin{thm}\label{T:Connection}
If $g$ is an extension of an $r$-graded sRC-manifold that makes the grading orthogonal, then there exists a unique connection $\nabla$ such that
\begin{itemize}
\item $g$ is metric compatible,
\item $V^j$ is parallel for all $j=0,\dots,r$,
\item $\tor(V^j,V^j)  \cap V^j = \{0\}$ for all $j=0,\dots,r$,
\item If $j\ne k$ and $X,Y \in V^j$ and $Z\in V^k$ then \[g(\tor(Z,X),Y) = g(\tor(Z,Y),X).\]
\end{itemize}
Furthermore if $X,Y $ are horizontal vector fields and $T$ is a vertical vector field then $\nabla X$, $\tor(X,Y)$, $\pi^0 \tor(X,T)$ are all independent of the choice of $g$ and the $r$-grading on $VM$.
\end{thm}

\begin{cor}\label{C:Extend}
If the sRC-manifold $M$ is equipped with a regular $r$-grading then for tangent vectors
$T_p \in V_p^j$ and $X_p \in H_p^0$ the projection $\pi^{j+1} \tor(T_p,X_p)$ depends only on $VM$ and the $r$-grading, not  the choice of $g$.
\end{cor}

\begin{proof}
If $X$ and $T$ are any vector field extensions of  the vectors $X_p ,T_p$ such that $X$, $T$ are sections of $H^0$ and $V^j$ respectively, then it is clear that
\[ \pi^{j+1} [T,X]_p = -\pi^{j+1} \tor(T,X)_p = -\pi^{j+1} \tor(T_p,X_p). \]
However from the definition of a regular grading, it is clear that the left hand side is independent of the metric extension.

\end{proof}

A direct consequence of the last corollary is that portions of the connection and torsion are invariant under $H$-isometries in the following sense.

\begin{cor}\label{C:Fundamental}
If $F \colon M \to \widetilde{M}$ is an sRC-isometry between sRC-manifolds, then for all horizontal vector fields $X,Y$,  vertical vector fields $T$ and arbitrary vector fields $A$
\begin{itemize}
\item $F_* \left( \nabla_A X \right) = \widetilde{\nabla}_{F_* A} F_* X$,
\item $F_* \tor(X,Y) = \widetilde{\tor}(F_*X,F_*Y)$,
\item $F_* \tor_H(T,X) = \widetilde{\tor}_H (F_* T,F_* X)$
\end{itemize}
\end{cor}


These connections are not torsion-free and this presence of torsion greatly complicates analysis on sRC-manifolds as compared to the Riemannian case. To obtain and optimize results, we shall use a variety of restrictions on the torsion.

\begin{defn}\label{D:Normality}
Let $\{E_i\}$ be any local orthonormal frame for $HM$ and $\{U_k\}$ any local orthonormal frame for $VM$, graded if $VM$ is graded. 
\begin{itemize}
\item An sRC-manifold is {\em $H$-normal} if $\tor(HM,VM) \subseteq VM$. This is independent of $g$ and choice of grading.
\item A metric extension is {\em $V$-normal} if $\tor(HM,VM) \subseteq HM$.
\item A metric extension is {\em strictly normal} if $\tor(HM,VM)=0$
\item  The {\em rigidity tensor} for $g$ is 
\[ \mathfrak{R}(A) = \sum\limits_i \aip{\tor(E_i,A)}{E_i}{} + \sum\limits_k \aip{\tor(U_k,A)}{U_k}{} \]
  The {\em rigidity vector} for $g$ is 
 \[ \widehat{\mathfrak{R}} = \sum\limits_k \mathfrak{R}(E_k)E_k + \sum\limits_i \mathfrak{R}(U_i)U_i. \]
 The sRC-manifold is {\em vertically rigid} if $\widehat{\mathfrak{R}} \in VM$ everywhere, {\em horizontally rigid} if $\widehat{\mathfrak{R}} \in HM$ everywhere and {\em totally rigid} if $\widehat{\mathfrak{R}} \equiv 0$. 
\end{itemize}
\end{defn}

For the remainder of this section, we shall work with the basic grading and assume that an extension to a Riemannian metric has been chosen . All results transfer to more complicated gradings without much effort.

With this connection in hand, we can now define what will be a key tool in studying the $H$-Killing fields of $M$.

\begin{defn}\label{D:Bil}
For a weak $H$-Killing field, $K$, we define a linear operator \[ \mB\colon TM \to HM\] by
\[ \mB (A) = \nabla_A K_H + \tor(K,A)_H = \left( \nabla_K A -[K,A] \right)_H \]
and a bilinear form on $TM$ by 
\[ \mB (A,B) = \aip{ \mB (A) }{B}{}  .\]
\end{defn}

\begin{defn}\label{D:Spaces}
We define $\mathscr{K}$ and $\mathscr{K}^*$ to be the Lie algebras of $H$-Killing fields and weak $H$-Killing fields respectively.  We further define
\begin{align*}
\mathscr{K}_V &= \{ K \in \mathscr{K} \colon K_H  \equiv 0 \}\\
\mathscr{K}_B &= \{ K \in \mathscr{K} \colon \mB \equiv 0 \}\\
\mathscr{K}_p &= \{ K \in \mathscr{K} \colon   K_p =0\}
\end{align*}
With a $*$ superscript denoting the weak versions of the same spaces. 
\end{defn}

To conclude this section, we list the basic properties of $\mB$ and the Lie algebras above.

\begin{lemma}[Elementary Properties]\label{L:EP}
If $K \in \mathscr{K}^*$ then
\begin{enumerate}
\item For $X \in HM$, 
\[ \nabla_X K_V = \tor(X,K)_V \]
\item $\mB$ is skew-symmetric on $HM$.
\item $\nabla_{K_H} K_V = - \tor(K_V,K_H)_V=-\tor(K,K_H)_V = - \tor(K_V,K)_V$
\end{enumerate}
If $K \in \mathscr{K}$ then additonally
\begin{enumerate}  \addtocounter{enumi}{3}
\item For $T \in VM$, 
\[ \nabla_T K_H = \tor(K,T)_H \]
\item $\mB(VM) =0$
\item $\mB$ is skew-symmetric on $TM$.
\item $ \nabla_{K_V} K_H = - \tor(K_H,K_V)_H  = -\tor(K,K_V)_H = - \tor(K_H,K)_H$
\end{enumerate}
\end{lemma}

\begin{proof} If $K \in \mathscr{K}^*$, then for a horizontal vector field  $X$
\begin{align*}
0 &= [K,X]_V = -\nabla_X K_V - \tor(K,X)_V 
\end{align*}
Next, if $X,Y \in HM$ then, since $\mathcal{L}_K \aip{}{}{} =0$, we see
\begin{align*}
 K \aip{X}{Y}{} &= \aip{[K,X]}{Y}{} + \aip{[K,Y]}{X}{}  \\
 &= \aip{ \nabla_K X - \nabla_X K - \tor(K,X) }{Y}{} \\ & \qquad + \aip{ \nabla_K Y - \nabla_Y K - \tor(K,Y) }{X} {} \\
 &= K \aip{X}{Y}{}  - \mB (X,Y) + \mB (Y,X) 
\end{align*}
This proves skew-symmetry on $HM$. Additionally, we see
\begin{align*}
0 &= [K,K_H]_V =  - \nabla_{K_H} K_V - \tor(K,K_H)_V
\end{align*}

Now if $K \in \mathscr{K}$ then we also see that for a vertical vector field $T$
\[ 0 = [T,K]_H = \nabla_T K_H - \tor(T,K_H)_H = \mB(T). \]
It is then clear that the skew-symmetry extends to all of $TM$.
Now we additionally see that
\begin{align*}
0 &=[K,K_V]_H = -\nabla_{K_V} K_H - \tor(K,K_V)_H \\
\end{align*}
\end{proof}  

\begin{cor}\label{C:vertelem}
If $T \in \mathscr{K}^*_V$ then $ \tor(T,HM)_H = 0 $. If $M$ is $V$-normal, then $T$ is $H$-parallel.
\end{cor}

\begin{proof} For a vertical weak $H$-Killing field $T$, we see that
\[ \mB[T]( \cdot) = \tor(T,\cdot)_H \]
is both symmetric and skew-symmetric on $HM$. Thus $\tor(T,HM)_H=0$. Thus is $M$ if $V$-normal then $\tor(T,HM)=0$. But then by Lemma \ref{L:EP} (a),  $\nabla_X T =0$ for all $X \in HM$.

\end{proof}

\section{The group $\text{Iso}(M)$ }\label{Iso}

Throughout this section we shall assume that $M$ is a complete with respect to the sub-Riemannian distance. 

We now turn to the task of determining the structure of the groups of isometries. In \cite{Strichartz} and \cite{StrichartzCor} , Strichartz established  that, under an additional condition known as strong bracket generation,  a weak isometry $F$ is determined by $F(p)$ and $(F_*)_p$ and that $\text{Iso}^*(M)$ is a Lie group.  This additional constraint does not appear to be necessary  for the Lie group result if we use the Kobayashi methodology outlined here. In this paper, we shall instead focus on $\text{Iso}(M)$. Under the relatively weak condition that $H$ bracket generates we shall establish that it is indeed a Lie group. We will then use the analytic machinery developed in the previous section to establish bounds on its dimension.

\begin{lemma}\label{L:global}
If $K \in \mathscr{K}^*$ then the  $1$-parameter subgroup of weak $H$-isometries generated by $K$ is globally defined on $M$.
\end{lemma}

\begin{proof} Suppose otherwise. Then there exists an integral curve of $K$ whose domain is not all of $\mathbb{R}^{}$. The Escape Lemma (see for example Lemma 17.10 in \cite{LeeSmooth})  implies that the image of $\gamma$ is not contained in any compact set. But if $\gamma(0)=p$ and $L_n = \sup\limits_{0 \leq t \leq n} d(\gamma(t),p)$, then $L_n = n L_1$. But this implies that $d(\gamma(t),p)$ is bounded by a constant multiple of $t$ and so as $M$ is complete $\gamma$ cannot escape all compact sets in finite time.

\end{proof}

Classically, a tool to study isometries is the structure of geodesics. Unfortunately, distance minimizing curves in sub-Riemannian geometry are poorly behaved in the sense that the exponential map is not a local diffeomorphism. For this reason, we shall introduce a different category of curves which we shall use to connect distant points on $M$.

\begin{defn}\label{D:Rules}
A {\em rule} on $M$ is a curve $\gamma \colon (-\epsilon,\epsilon) \to M$ such that
\[  \dot{\gamma} \in H_\gamma M, \qquad  \nabla_{\dot{\gamma}} \dot{\gamma} = 0 \]
We say that two points $p,q \in M$ are {\em rule related} if they can be joined by a finite concatenation of rules. The manifold $M$ is {\em traversable} if every pair of points in $M$ is rule related.
\end{defn}

Since Corollary \ref{C:Fundamental} states that smooth $H$-isometries preserve the connection, they must also map rules to rules. We remark that this need not be true for weak $H$-isometries and indeed this marks a key distinction between the weak and strong cases.

 Our next goal is to study how a manifold is connected by rules. To proceed, we shall need a technical lemma.

\begin{lemma}\label{L:CMZ}
If $HM$ bracket generates at every point and $\Sigma \subset M$ is an embedded submanifold, then the characteristic set 
\[ C(\Sigma) = \{ p \in \Sigma \colon H_p M \subset T_p \Sigma\}\]
is a closed, nowhere dense set of $\Sigma$. 
\end{lemma}

\begin{proof} Suppose $p$ is a limit point of $C(\Sigma)$.  Near $p$, let $N=(N_1,\dots,N_k)$ be a framing for the normal bundle of $\Sigma$ with respect to any metric extension. Then near $p$, $C(\Sigma) = \{N_H =0\}$. This is a closed condition and so $p \in C(\Sigma)$.  Now, if the complement of $C(\Sigma)$ in  $\Sigma$ contains an open set, then on that open set $HM$ can only bracket generate $T\Sigma$.

\end{proof}

In fact, far stronger statements can be shown. For details in the hypersurface case, see the appendix of \cite{HP4}, or the results of \cite{Derridj}.

\begin{thm}\label{T:Traversable}
If $M$ is a connected manifold such that $HM$ bracket generates at every point, then $M$ is traversable.
\end{thm}

\begin{proof}
Since rule relation is an equivalence relation and $M$ is connected, it suffices to show that for every $p \in M$, there is an open set  containing $p$ such that every point in the set is rule related to $p$.

The idea is to begin by considering the immersed submanifolds $R_p$ consisting of all rules emanating from a point $p\in M$. Locally, these submanifolds will yield a foliation with leaves of dimension $\dim HM$. The proof will proceed inductively by showing that the leaves of any traversable foliation can be pieced together to provide a traversable foliation of a neighborhood of $p$ with leaves of dimension one larger. Repeating this process will eventually produce a foliation by traversable leaves of the same dimension as the manifold. Such a foliation must necessarily consist of a single leaf and so yields the desired open neighborhood.

Fix $p \in M$ and suppose we have a local coordinate chart $U \times W$ with coordinates $(u^1,\dots,u^k,w^1,\dots,w^m)$ such that $p=(0,\dots,0)$ and each submanifold $F_c =\{ w =c\}\subset U \times W$ is traversable. The manifolds $F_c$ form a foliation by dimension $k$ traversable submanifolds and the coordinates above are local slice coordinates for this foliation.

From Lemma \ref{L:CMZ}, there must exist $(q,0)  \in F_0$ such that $H_qM \nsubseteq T_q F_0$. Thus after a linear change in the $w$ coordinates, we can assume that  there is a rule $\gamma(t)$ with $\gamma(0)=q$ and $\gamma^\prime(0)$ has non-zero component in the $w^1$ coordinate.  Let $\Gamma$ be any extension of $\gamma^\prime(0)$ to a smooth horizontal vector field near $(q,0)$. Define a smooth function $G$ defined on small neighborhood of $0 \in \mathbb{R}^m$ by
\[ G(t, w^2,\dots,w^m) =\pi_W \circ \gamma_{(w^2,\dots,w^m)}(t) \]
where $\pi_W$ represents projection onto $W$ and $ \gamma_{(w^2,\dots,w^m)}(t)$ is the rule  emanating from the point with coordinates $(q,0,w^2,\dots,w^m)$ with initial tangent vector $\Gamma_{(q,0,w^2,\dots,w^m)}$.

It is clear that $(G_*)_{(q,0)}$ has full rank so, by the inverse function theorem, there is a smooth inverse $G^{-1}$ defined on  some open set $0 \in \tilde{W} \subset W$.  We split $G^{-1}$ into the first and remainder coordinates by $G^{-1} = (G^{-1}_1, G^{-1}_{>1})$. 

We now define a foliation of $U \times \tilde{W}$ by dimension $k+1$ submanifolds by setting
\[ \tilde{F}_{b} = \{ (u,w) \in U \times W \colon  G^{-1}_{>1} (w) = b\}\]
for $b \in G_{>1}^{-1}(\tilde{W})$.
To complete the argument we must argue that each $\tilde{F}_b$ is traversable.

First note that $G(0,w^2,\dots,w^m)= (q,0,w^2,\dots,w^m)$ and so each $\tilde{F}_b$ contains a point  $b^* = (q,0,b) \in U \times \tilde{W}$. We shall argue that all other points in $\tilde{F}_b$ can be connected to $b^*$ by a finite number of rules.

Suppose $(r,s) \in \tilde{F}_b$. Then $G^{-1}(s) =  (t_0,b)$ for some $t_0$, which implies that $s = \pi_W \circ \gamma_{b}(t_0)$.  But  then  $(r,s),\gamma_b(t_0) \in F_s$. Since our inductive assumption is that the level sets of $F$ are traversable, there must be a finite sequence of rules connecting $(r,s)$ to $\gamma_b(t_0)$. Now $\gamma_b(t)$ itself is a rule connecting $\gamma_b(t_0)$ with $b^*$. Hence $(r,s)$ can be connected to $b^*$ using a finite number of rules. Thus the leaves $\tilde{F}_b$ are themselves traversable and the inductive step is complete.

\end{proof}

An important consequence is the  following lemma.

\begin{lemma}\label{L:IsoBasic}If $M$ is connected and $F\colon M \to M$ is an $H$-isometry such that $F(p)=p$ and
\[  {F_*}_{|H_p} = \text{Id}_{H_p} \]
then $F = \text{Id}_M$.
\end{lemma}

\begin{proof}
Let $\gamma$ be any rule emanating from $p$. Since $F$ is an $H$-isometry with $F(p)=p$, we immediately get that $F \circ \gamma$ is also a rule emanating from $p$. However 
\[  (F \circ\gamma)^\prime(0)= (F_*)_{|p}(\gamma^\prime(0)) = \gamma^\prime(0),\] so $F \circ \gamma = \gamma$.

Now suppose $q$ is any point on $\gamma$ and that $v \in H_q M$. Clearly we have $F(q)=q$. Let $Y$ be the parallel horionztal vector field along $\gamma$ such that $Y_q =v$.  Then
\[ 0 = {F_*}_{|\gamma(t)} ( \nabla_{\gamma^\prime} Y) = \nabla_{{F_*}_{|\gamma(t)} (\gamma^\prime) } {F_*}_{|\gamma(t)} (Y) =  \nabla_{\gamma^\prime}  {F_*}_{\gamma(t)} (Y)\]
Thus ${F_*}_{\gamma(t)} Y$ is also parallel along $\gamma$.  However ${F_*}_{|p} Y = Y_p$ so we must have ${F_*}_{|q}(v)= {F_*}_{|q}(Y_q) =Y_q=v$.

Since $M$ is traversable by Theorem \ref{T:Traversable}, an easy induction argument now shows that $F = \text{Id}_M$.

\end{proof}

From this we can establish the fundamental theorem that establishes that $\mathscr{K}$ is finite dimensional and can be determined by more restricted information at a single point $p\in M$ than is required by Strichartz for the weak isometries.

\begin{thm}\label{T:Basic}
The map $ \chi_p \colon \mathscr{K}  \to T_p M \times \mathfrak{skew}(H_p)$ defined by
\[ \chi_p(K) = \left(K_{|p},  \left( \mB \right)_{|p} \right)\]
is injective.
\end{thm}

\begin{proof}
Suppose $p \in M$ and that $K$ is an $H$-Killing field such that $K_{|p}=0$ and  $\left(\mB \right)_{|p}=0$. Since $\chi_p$ is linear, it suffices to show that $K \equiv 0$.

Let $Y$ be a horizontal vector field near $p$. Since $[K,Y]$ must be horizontal, we see that $ [K,Y] = \nabla_K Y - \mB(Y)$. The assumptions on $K$ then imply that 
$[K,Y]_{|p} =0 $. This means that the one parameter family of $H$-isometries generated by $K$ all act as the identity on $H_pM$. By Lemma \ref{L:IsoBasic},  this family consists only of the identity map and so $K \equiv 0$.

\end{proof}

\begin{cor}\label{C:KFD}
The Lie algebra $\mathscr{K}$ is finite dimensional.
\end{cor}

We now recall the classical theorem
\begin{thm}[Palais]\label{T:Palais}
 Let $G$ be a group of differentiable transformations of a manifold $M$ and $S$ is the set of all vector fields that generate global $1$-parameter subgroups of $M$. If $S$ generates a finite dimensional Lie algebra of vector fields on $M$, then $G$ is a Lie transformation group and $S$ is the Lie algebra of $G$.
\end{thm}
This theorem was originally due to Palais \cite{Palais} but this formulation is due to Kobayashi and appears as Theorem 3.1 in \cite{KobayashiT}. The reader is also referred to the work of D. Montgomery \cite{DMont} for an alternative approach.
 
Since, from Lemma \ref{L:global}, we know that all $H$-Killing fields are complete, we can immediately deduce the following.
\begin{cor}\label{C:LG}
$\text{Iso}(M)$, the group of smooth $H$-isometries,  is a Lie group.
\end{cor}

Furthermore we get the following simple corollaries of Theorem \ref{T:Basic}.

\begin{cor}\label{C:pureV}
If $K \in \mathscr{K}_V$ then either $K \equiv 0$ or $K$ is non-vanishing. 
\end{cor}

\begin{cor}\label{C:HNorm}
If $\dim \mathscr{K}_V = \dim VM$ then $M$ is $H$-normal and admits a strictly normal metric extension.
\end{cor}

\begin{proof}
From Theorem \ref{T:Basic}, we immediately see that the purely vertical $H$-Killing fields form a global frame for $VM$. That $M$ is $H$-normal, then follows immediately from Corollary \ref{C:vertelem}. If we define a metric extension by declaring the purely vertical $H$-Killing fields to be an orthonormal frame for $VM$, then it is easy to verify that for $T \in \mathscr{K}_V$, we have $\nabla_X T = 0 = [T,X]_V$  for all horizontal $X$. Thus the extension is strictly normal.

\end{proof}

For our final corollary of Theorem \ref{T:Basic}, we obtain a crude upper bound on the dimension of $\mathscr{K}$.

\begin{cor}\label{C:dim} 
If $\dim(HM)=k$ and $\dim (VM)=m$ then 
\[ \dim \mathscr{K}  \leq m +k + \frac{k(k-1)}{2} = m + \frac{k(k+1)}{2}. \]
\end{cor}

\section{Torsion bounds on $\dim \text{Iso}_p(M)$}\label{Tor}

In this section, we shall  work at a fixed $p \in M$ and use the torsion structure of an sRC-manifold to refine our estimates on the dimension of the isotropy group at $p$, $\text{Iso}_p(M)$. 

We first note that if $K \in \mathscr{K}_p$ then $[K, \cdot]$ makes sense as a pointwise operator $T_pM \to T_p M$. If $A$ is any vector field then, since $K$ vanishes at $p$, 
\[ [K,A]_{|p} = \left(- \nabla_A K  -\tor(K,A) \right)_p.\]
The expression within parenthesis depends tensorially  on $A$  and thus depends solely on the value of $A$ at $p$.

Now fix $V^1_p = H^1_p \cap V_pM$. For any $F \in \text{Iso}_p(M)$, since both $H^1_p$ and $V_pM$ are invariant subspaces under $F_*$, we see that $V^1_p$ is invariant also. A key tool in our estimation will be subsets of $V^1_p$ that are also invariant under $\text{Iso}_p(M)$.
 
\begin{defn}\label{D:indicatrix}
For a pair of subspaces $E_1,E_2 \leq  H_pM$, define the {\em torsion indicatrix} of $(E_1,E_2)$ by
\[ \Omega_p^1(E_1,E_2) = \{ \tor(X_1,X_2) \colon X_i \in E_i, \|X_i\|=1 \} \subseteq V_p^1.\]
For simplicity of notation, we shall use $\Omega_p^1(E) = \Omega_p^1(E,E)$.
\end{defn}

We remark that torsion can be viewed as a linear map $\mathcal{T}  \colon \wedge^2 H_p \to V_p^1 $. Then $\Omega_p^1(E_1,E_2)$ can be viewed as the image of product of unit spheres  $\mathbb{S}^{}(E_1) \times \mathbb{S}^{}(E_2)$ under the composition $\mathcal{T} \circ \wedge $ where
\[ \wedge \colon E_1 \times E_2 \to \wedge^2 H_p \]
is the obvious wedge product.

It is  useful to note the following result.
\begin{lemma}\label{D:star}
For any subspace $E \leq H_pM$, $\Omega_p^1(E)$ is compact, symmetric and star-shaped about $0 \in V^1_p$.
\end{lemma}

\begin{proof}
The compact and symmetric properties are trivial and left to the reader.  Now suppose that $X,Y \in E$ with $\|X\|=\|Y\| =1$.  Thus $\tor(X,Y)$ represents a generic point in $\Omega_p^1(E)$.  Now for any $0\leq t \leq 1$, we note that there must exist $s \in \mathbb{R}$ such that $ tY + s X$ is a unit length vector in $E$. Since
\[ t \tor(X,Y) =\tor(X, tY + s X )  \in \Omega_p^1(E)\]
we see that $\Omega_p^1(E)$ is star-shaped about $0$.
\end{proof}

\begin{defn}\label{D:stab}
For a subset of a vector space, $\Omega \subset V$, we define the {\em stabilizer of $\Omega$ relative to $V$} to be
\[ \text{St}(\Omega,V) = \{ L \in GL(V) \colon L(\Omega)=\Omega \}.\]
If the vector space is clear from context we shall drop it from the notation and use $\text{St}(\Omega)$ instead.
\end{defn}

If $V$ is a finite dimensional vector space, then $\text{St}(\Omega,V)$ is a closed Lie subgroup of $GL(V)$. Our key example of a stabilizer is $\text{St}(\Omega_p^1) = \text{St}(\Omega_p^1 , V_p^1)$. This is due to the trivial fact that if $F \in \text{Iso}_p(M)$ then $(F_*)_{|V_p^1} \in \text{St}(\Omega_p^1)$.

\begin{lemma}\label{L:dimsplit}
If we define a closed Lie subgroup of $\text{Iso}(M)$ by
\[ \text{Iso}_p^\bullet (M) = \{ F \in \text{Iso}_p(M) \colon (F_*)_{|V_p^1} = \text{Id} \}\]
then 
\[ \dim \text{Iso}_p(M) \leq \dim \text{Iso}_p^\bullet(M) + \dim \text{St}(\Omega_p^1).\]
\end{lemma}

\begin{proof} The map $F \mapsto (F_*)_{|V^1_p}$ is a Lie group homomorphism from $\text{Iso}_p(M)$ to $\text{St}(\Omega_p^1)$. Let $L$ be the associated Lie algebra homomorphism. It is then clear that $\ker(L) \subseteq \mathscr{K}^\bullet_p$. The result then follows from elementary linear algebra.

\end{proof}

The relationship between the operator $\mB$ and the torsion in the following lemma is the key to estimating the dimension of $\text{Iso}_p^\bullet(M)$.

\begin{lemma}\label{L:vd2}
 If $K \in \mathscr{K}^\bullet_p$ and  $X,Y \in H_p M$
 \[ \tor(X,\mB(Y)) + \tor(\mB(X), Y) = 0.\]
 \end{lemma}
 
 \begin{proof} For any smooth section $T$ of $V^1$, we must have $[K,T]_p=0$. Thus if we extend $X,Y$ to horizontal vector fields we have
 \begin{align*}
 0 & = [K, \tor(X,Y)]_p = \tor([K,X],Y)_p + \tor(X,[K,Y]_p) \\
 &= \tor( \mB(X) -\nabla_K X ,Y) + \tor (X,\mB(Y) -\nabla_K Y ).
 \end{align*}
 Since $K_p=0$ the result follows immediately.
 
  \end{proof}

As an interpretation of this result, choose an orthonormal frame for $H_pM$. Identify $\mB$ with the matrix of the skew-adjoint operator  $\mB: H_p M \to H_pM$  and let $\breve{T}$ be the matrix of the $V^1_p$-valued skew-symmetric bilinear form $\tor \colon H_p M\times H_p M \to V_p^1$. Then as matrices
\[ [\mB,\breve{T}]=0.\] 
Since $\breve{T}$ is a vector-valued matrix, this can provide a considerable restriction on the dimension of $\text{Iso}_p^\bullet(M)$.

To use the full power of Lemma \ref{L:dimsplit}, we must now develop methods for studying $\text{St}(\Omega^1_p)$. These are typically difficult to compute directly, but as we shall see, we can use decomposition methods to infer information on the structure of $\Omega^1_p$.

\begin{defn}\label{D:kernel}
The {\em torsion-kernel} at $p$, $\widetilde{H}^0$,  is the maximal subspace $\widetilde{H}^0 \subseteq H_pM$  with the property that
\[ \Omega_p^1(\widetilde{H}^0, H_pM) = 0.\]
We say that $M$ is {\em strongly non-integrable} at $p$ if $\widetilde{H}^0 = 0$.
\end{defn}

\begin{defn}\label{D:torsiondec}
A {\em weak torsion decomposition} at $p$ is an orthogonal decomposition $H_pM = \widetilde{H}^0 \oplus \bigoplus\limits_{i=1}^m\widetilde{H}^i$ 
such that $\Omega^1_p (\widetilde{H}^i,\widetilde{H}^j) =0$ if $i \ne j$. 
A weak torsion decomposition is {\em strong} if additionally there is a vector space decomposition $V^1_p =\bigoplus\limits_{i=1}^m\widetilde{V}^i$ such that $\Omega_p^1(\widetilde{H}^i) \subseteq \widetilde{V}^i, i=1,\dots,m.$
\end{defn}

For subsets $\Delta_1,\dots,\Delta_m$ contained in some vector space $V$, we define the convex sum
\[ \mathscr{C} \sum\limits_{i=1}^m \Delta_m  = \left\{  \sum\limits_{i=1}^m a_i d_i \colon d_i \in \Delta_i, a_i \geq 0\text{ for } i=1,\dots,m \text{ and } \sum\limits_{i=1}^m a_i \leq 1 \right\}.\]
Now we recall that each $\Omega_p^1(\widetilde{H}^i)$ is star-shaped and symmetric under multiplication by $\pm 1$. From a simple application of the Cauchy-Schwartz inequality, we can obtain the following result.

\begin{lemma}\label{L:torsiondec}
For any weak torsion decomposition, 
\[ \Omega_p^1(H_pM)=  \mathscr{C}\sum\limits_{i=1}^m \Omega_p^1(\widetilde{H}^i).\]\end{lemma}

\begin{proof}
For unit length vectors $X,Y \in H_p$ we can decompose $X= a^i X_i$, $Y =b^j Y_j$ with $X_i,Y_i \in \widetilde{H}^i$ and unit length. Then
\[ X \wedge Y = \sum\limits_{i=1}^m a^i b^i X_i \wedge Y_i  \qquad \text{mod } \ker(\mathcal{T}).\]
By replacing $X_i$ by $-X_i$ as necessary, we can suppose each $a^ib^i \geq 0$. But  then
\[ 0 \leq \sum a^i b^i \leq \frac{1}{2} \sum (a^i)^2  + \frac{1}{2} \sum (b^i)^2 \leq 1.\] 
We can also clearly obtain express any convex sum 
\[ t^i X_i \wedge Y_i, \qquad t^i \geq0 , \, \sum t^i \leq 1\]
in the form $(a^i X_i ) \wedge (b^j Y_j)$ modulo $\ker(\mathcal{T})$ with $\sum (a^i)^2 = \sum(b^j)^2=1$.

 \end{proof}
 
 \begin{cor}\label{C:2dsplit}
 If $M$ admits a weak torsion decomposition at $p$ with each $ \dim \widetilde{H}^i \leq 2$ for $i=1,\dots,m$, then $\text{St}(\Omega_p^1)$ is discrete and $\dim \text{Iso}_p(M) = \dim \text{Iso}_p^\bullet(M)$.
 \end{cor}
 
 \begin{proof}
 If each torsion block other than the torsion kernel has dimension at most $2$, then each non-zero $\Omega(\widetilde{H}^i)$ consists of a closed line segment. The convex sum of a finite number of closed line segments will always be polyhedral in nature and so will have discrete symmetry group.
 
 \end{proof}

If we have a strong torsion decomposition where the component pieces of the sums are in the summands $\widetilde{V}^i$ which intersect trivially, then we can get a stronger result.  We note in particular that with a strong decomposition for each $i>0$, 
\[\Omega_p^1(\widetilde{H}^i) = \Omega_p^1(H_pM) \cap \widetilde{V}^i.\]

To make use of this structure, we shall need a lemma from basic linear algebra.

\begin{lemma}\label{L:linmap}
Suppose that the finite dimensional vector space $W =W_1 \times W_2$ and that  for $i=1,2$, $\Omega_i\subset W_i$ is a compact set that is star-shaped and symmetric about $0$ such that $\text{span}(\Omega_i)=W_i$. Let $\Omega$ be the convex sum of $\Omega_1$ and $\Omega_2$ in $W$ . 
 
 For any $L$ in the connected component of the identity in $\text{St}(\Omega)$, if we identify $W_1,W_2$ with subspaces of $W$ in the natural way, then $L(W_i) = W_i$ and $L_{|W_i} \in \text{St}(\Omega_i)$ for $i=1,2$.
 \end{lemma}
 
 \begin{proof} Define an extremal point of a subset $E\ne \emptyset$ of a vector space $V$, to be a point $x \in E$ such that there is no affine linear embedding of $(-\epsilon,\epsilon)$ into $E$ such that $0$ is sent to $x$. Let $E^e$ denote the set of extremal points of $E$. Next we claim that if $E$ is compact then, $E \subset \text{span}(E^e)$. To do this impose an arbitrary inner product $\aip{\cdot}{\cdot}{}$ on $V$. Any point of $E$ maximizing the associated norm, must then be an extremal point. Let $Y = \text{span}(E^e)$ and let $\pi^\bot \colon V \to Y^\bot$ be the projection onto its orthogonal complement. If $E \nsubseteq Y$, then $\mu= \max \{ \|\pi^\bot(x) \| \colon x \in E\} >0$ and $E_\mu = \{ x \colon E \colon \| \pi^\bot(x)\| = \mu\}$ is a non-empty, compact set. Clearly any affine linear embeddings of intervals into $E$ mapping $0$ onto a point of this set must be contained in an affine linear subspace   parallel  to  $Y$. But then any extremal point of $E_\mu$, viewed as a compact non-empty subset of this affine linear subspace, must also be an extremal point of $E$. Since such points would have to exist, we have a contradiction and $E \subseteq Y=  \text{span}(E^e)$.
  
Now returning to our set up, it is easy to check that since each $\Omega_i$ is star-shaped that 
\[\Omega^e =\left(  \Omega_1^e \times \{0\} \right) \cup  \left( \{0\} \times  \Omega_2^e\right).\]
 As $(0,0)$ is not an extremal point of $\Omega$, the points $(x_1,0)$ and $(0,x_2)$ are in different path-components of the set of extremal points of $\Omega$. Now clearly $L \in GL(\Omega)$ restricts to a homeomorphism $\Omega^e \to \Omega^e$. Therefore any $L$ is the connected component of the identity must induce homeomorphisms $\Omega_i^e \to \Omega_i^e$ for $i=1,2$. As $\Omega_i^e$ spans $W_i$ we must therefore have $L(W_i)=W_i$. Now with $W_i$ viewed as a subspace of $W$ in the natural way, $\Omega_i = \Omega \cap W_i$ and so $L(\Omega_i) =\Omega_i$.
 
 \end{proof}
 
  \begin{cor}\label{C:stc} $\dim \text{St}(\Omega) \leq \dim \text{St}(\Omega_1) + \dim\text{St}(\Omega_2)$.
 \end{cor}

 The lemma and its corollary extend to convex sums of more than two subsets in the obvious way. For our purposes, the primary use of the lemma is in the following result.
  
 \begin{thm}\label{T:decomposition}
Suppose $M$ admits a strong torsion decomposition at $p$. For any $F \in \text{Iso}_p(M)$, the push-forward $(F_*)_{|p}$ preserves the splittings of both $H_pM$ and  $V^1_p$.
 \end{thm}
 
 \begin{proof} The first key observation is that for any $F \in \text{Iso}_p(M)$ and $X,Y \in H_pM$ we have
 \[ F_* \tor(X,Y) = \tor(F_* X,F_*Y) .\]
 From this it is trivial to show that $F_* \widetilde{H}^0 = \widetilde{H}^0$ and hence that $F_*$ preserves $(\widetilde{H}^0)^\bot = \widetilde{H}^1 \oplus \dots \oplus \widetilde{H}^m$.
 
 Now since $H^1_p$ is bracket-generated by $HM$ at $p$, we see that each $\Omega^1(\widetilde{H}^i)$ must span $\widetilde{V}^i$. Thus we can apply Lemma \ref{L:linmap} to see that $F_*$ preserves the splitting on $V^1_p$. 
 
Suppose that $X \in \widetilde{H}^i$ with $i\ne 0$ and that $\tilde{X} = \pi^j F_* X$ for some $0\ne j \ne i$. Since $\tilde{X} \in \widetilde{H}^j$ is not in the torsion-kernel, there must be some $Y \in \widetilde{H}^j$ such that $\tor(\tilde{X},Y) \ne 0$. Then
 \[  F_* \tor(X,F^{-1}_* Y ) = \tor( \tilde{X}, Y)  \in \widetilde{V}^j \backslash \{0\}. \]
 But since $\tor(X,F^{-1}_* Y) \in \widetilde{V}^i$, this is a contradiction. 
 
 \end{proof}

The consequence of this theorem is that when looking for torsion restrictions on the dimension of $\dim \mathscr{K}_p$, we can study each block of a strong torsion decomposition separately. In particular, if $h^i = \dim \widetilde{H}^i$ we can immediately improve Corollary \ref{C:dim} by noting
\[ \dim \mathscr{K}_p \leq \sum\limits_{i=0}^m \frac{h^i(h^i-1)}{2} .\]
However, using Lemma \ref{L:vd2} on each block will typically reduce this even further. We summarize this discussion in the following theorem, the proof of which is now trivial.

\begin{thm}\label{T:tbound}
Suppose that  $H_pM$ admits a strong torsion decomposition $H_pM = \tilde{H}_0\oplus \bigoplus\limits_{i=1}^m  \tilde{H}_i$ with $\dim \tilde{H}_i=h_i$, $i=0,\dots,m$.  Let 
\[ \mathfrak{h}_i =\{ A \in \mathfrak{skew}(\tilde{H}_i) \colon [A,\breve{T}_{ii} ]=0 \}\]
where $\breve{T}_{ii}$ denotes the $\tilde{H}_i$ block on the diagonal of the matrix $\breve{T}$ when written with respect to an orthonormal frame respecting the strong torsion decomposition.
Then
\begin{align*} \dim \mathscr{K}^\bullet_p &\leq \frac{h_0(h_0-1)}{2}  + \sum\limits_{i=1}^m  \dim \mathfrak{h}_i,  \\
 \dim \text{St}(\Omega_p^1) &\leq \sum\limits_{i=1}^m  \dim \text{St}(\Omega^1_p(\tilde{H}_i)),\\
\dim \text{Iso}_p(M) &\leq  \dim \mathscr{K}^\bullet_p+\dim \text{St}(\Omega_p^1).
\end{align*}
\end{thm}

As a simple application, we can look at the structure of product manifolds and obtain the following.

\begin{cor}\label{C:product}
If for $i=1,2$, $M_i$ is an sRC-manifold that is strongly non-integrable at $p_i$, then
\[ \mathscr{K}_{(p_1,p_2)} (M_1 \times M_2) =\mathscr{K}_{p_1} (M_1) \times \mathscr{K}_{p_2}(M_2).\]
\end{cor}
Without the strongly non-integrable condition, it is possible that there will be interaction between the torsion-kernels.

Next we note that Theorem \ref{T:tbound} and the above discussion only take into account torsion as a map $HM \times HM \to V^1$.  If the sRC-manifold $M$ is not $H$-normal then we get alternative torsion restrictions on $\dim \text{Iso}_p^\bullet$. To see this, define for any  $U \in V_p^1$ a linear map $\mathcal{T}_U^0 \colon H_p \to H_p$ by $ \mathcal{T}_U^0(X)  =\pi^0 \tor(U,X)$. This map is independent of the choice of metric extension and vanishes identically if $M$ is $H$-normal.

Now for $K \in \mathscr{K}_p^\bullet$ we see
\begin{equation}\label{E:commutes}
\begin{split}
 \mB \circ \mathcal{T}_U^0(X)& = - [K,\mathcal{T}_U^0(X)]= -[ K,\pi^0 \tor(U,X)]  \\
 & = -\pi^0 \tor(U,[K,X]) =\mathcal{T}_U^0 \circ  \mB (X). \end{split}\end{equation}
 Now each $\mathcal{T}_U^0$ is a self-adjoint operator on $H_p$ and so there is a complete decomposition of $H_p$ into orthogonal eigenspaces of $\mathcal{T}_U^0$. The commutation property of \eqref{E:commutes} then implies that $\mB$ must preserve this eigenspace decomposition. 
 
Unfortunately however, if $\dim V^1_p >1$, the maps $\mathcal{T}^0_U$ and $\mathcal{T}^0_{\tilde{U}}$ for different $U,\tilde{U}$ typically do not commute. This means it typically will not be possible to consider mutual eigenspace decompositions across all such operators $\mathcal{T}^0_U$. 

Suppose that  $U \in V^1_p$ and that $H_p M = \bigoplus\limits_{i=1} ^m E_i $ where each $E_i$ is an eigenspace of $\mathcal{T}^0_U$ with $\dim E_i =r_i$.  A first crude use is to note that
\begin{equation}
\dim \mathscr{K}^\bullet \leq \sum\limits_{i=1} ^m  \frac{r_i(r_i-1)}{2} .
\end{equation}
Using Lemma \ref{L:vd2}, we can improve this. Consider an orthonormal basis for $H_pM$ respecting the eigenspace decomposition. Writing $\breve{T}$ and $\mB$ as matrices with respect to this matrix, we can break the matrices into block components $\breve{T}_{ij}$ and $(\mB)_{ij}$ corresponding to the various eigenspaces. The block components of $\mB$ must lie on the diagonal. Hence, since $[\mB,\breve{T}] = 0$ we must have $(\mB)_{ii} \breve{T}_{ij} = \breve{T}_{ij} (\mB)_{jj}$ for all $i,j$.

We summarize this discussion as the following alternative to Theorem \ref{T:tbound}.

\begin{thm}\label{T:tbound2}
With the notation as above,  
\[ \dim \mathscr{K}_p^\bullet \leq \dim \{ B = \text{diag}(B_1,\dots,B_m) \colon  B_i \in \mathfrak{skew}(E_i), B_i \breve{T}_{ij} = \breve{T}_{ij} B_j \}.\]
\end{thm}

It should be remarked here that the operators $\mathcal{T}^0_U$ will not typically preserve a strong torsion decomposition, so this must be regarded as an alternative rather than complementary decomposition method.

\subsection{sRC-manifolds of large step size}

For regular sRC-manifolds of large step size, the above results can be applied verbatim, but they do not take into account any of the higher order torsion of the manifold.  Non-regular $H$-isometries typically will not preserve any higher order constructions but regular $H$-isometries will. The variety of different types of structure that can appear make it difficult to produce any general results. However, we shall outline one way in which the ideas of this section can be applied in the higher step case.

For the remainder of this section, we shall assume that $M$ admits a regular grading. We shall then explore the additional constraints on the regular $H$-isometry group $\text{Iso}^R_p(M)$. We can now assume that each $V_p^j$ is invariant under $F_*$ rather than just $V^1_p$. If $M$  admits a strong torsion decomposition, then our results can be applied to each block separately.

 \begin{defn}\label{D:TOR}
 We define the {\em step $m$ torsion} to be the $m+1$ tensor defined inductively by $\Tor[1]=\tor$ and
\[  \Tor[m](Z_1,\dots,Z_{m+1}) = \tor( Z_1,\Tor[m-1](Z_2,\dots,Z_{m+1})).\]
 \end{defn}
 \begin{lemma}\label{L:torj}
 For a regular sRC-manifold, the restricted multi-linear map \[\mathcal{T}^m = \pi^m \circ \Tor[m] \colon \prod\limits_{k=1}^{m+1} HM \to V^{m} \] depends only on the choice of grading $V_1,\dots,V_{r-1}$ not the metric extension.
  \end{lemma}
 
 The {\em torsion indicatrix at point $p$ of step $m+1$} is defined by
 \[ \Omega^m_p = \{\mathcal{T}^m(X_1,\dots,X_{m+1} ) \colon X_i \in H_p M, \|X_i\|=1\} \subset V^m_p. \]
 The {\em full torsion indicatrix} is then 
 \[ \Omega_p = \prod\limits_{m=1}^r \Omega^m_p. \]
 It is clear that
\[  \text{St}(\Omega_p) = \bigoplus\limits_{m=1}^r \text{St}(\Omega^m_p, V^m_p).\]
A regular sRC-manifold is said to be {\em vertically discrete at step $j+1$} at a point $p \in M$   if $\text{St}(\Omega^j_p,V^j_p)$ is a discrete group and {\em vertically discrete} if the full group $\text{St}(\Omega_p)$ is discrete. 

The following two lemmas are trivial and the proofs are left to the reader.

\begin{lemma}\label{L:FR}
If $F \in \text{Iso}^R_p(M)$ then $(F_*)_{|V_pM} \in \text{St}(\Omega_p)$.
\end{lemma}
  
 \begin{lemma}\label{L:1d}
 If   $\dim V^j=1$ for any $j=1,\dots,r$ then $M$ is vertically discrete at step $j+1$ everywhere.
 \end{lemma}
 
For manifolds that are vertically discrete at two successive levels, it is possible to greatly reduce the dimension of the vector space upon which $\mB$ acts as  a skew-symmetric operator. 
For $U \in V_pM$ define maps
\[ \mathcal{T}^m_U \colon H_p M \to V^m_p, \qquad \mathcal{T}^m_U = \pi^{m} \tor(U,X). \]

\begin{lemma}\label{L:vdm} Suppose that $K \in \mathscr{K}_p^R$. If $M$ is vertically discrete at steps $m+1$ and $m+2$ at $p$ then for any $U \in V^m_p$ 
\[ \text{range}\left( \mB \right)  \subseteq \ker \left( \mathcal{T}^{m+1}_U \right).\]
\end{lemma}

\begin{proof}
These follow from  very similar computations to that of Lemma \ref{L:vd2}. First since $M$ is vertically discrete at step $m+2$, we see
\begin{align*} 
0 &= [K, \pi^{m+1} \tor(U,X) ]_{|p}   = \pi^{m+1} [ K, \tor(U,X) ]_{|p}  \\
&= \pi^{m+1}  \tor([K,U],X)_{|p} + \pi^{m+1} \tor(U,[K,X])_{|p}
\end{align*}
As  $M$ is vertically discrete at step $m+1$, the first term on the last line vanishes. Hence at $p$, $\mathcal{T}^{m+1}(\mB(X) )=0 $. 
\end{proof}

From this we obtain our main theorem for regular higher step manifolds.
 
 \begin{thm}\label{T:higher}
 Suppose $M$ is vertically discrete at steps $m+1$ and $m+2$ at $p$. If
 \[ L = \bigcap\limits_{U \in V_p^m} \ker \mathcal{T}^{m+1}_U \]
 then at $p$, $L^\bot \subseteq \ker{\mB}$ and $\mB$ restricts to a map $L\to L$.
 \end{thm}
 
 \begin{proof}  It follows immediately from Lemma \ref{L:vdm} that  $\text{range}(\mB) \subseteq L$ at $p$. Since $\mB$ is skew-adjoint as a map $H_pM \to H_pM$, we have $\ker(\mB) = \text{range}(\mB)^\bot \supseteq L^\bot$. 
  
 \end{proof}
 
 From this we can easily  obtain the following dimension bound.
 
 \begin{cor}\label{C:higher}
 Under the same conditions as Theorem \ref{T:higher}, 
then
 \[ \dim \mathscr{K}_p \leq \frac{1}{2} \dim L (\dim L -1).\]
 \end{cor}

 \section{Curvature constraints}\label{Neg}
 
 In classical Riemannian geometry, the curvature of the manifold imposes both local and global conditions on the isometries and Killing fields. In this section we shall explore how this theory generalizes to the sub-Riemannian setting. First we shall look at the pointwise relationship between the curvature of the canonical sub-Riemannian connection and the bracket structure of $H$-Killing fields. Next we shall define a sub-Riemannian analogue of the Ricci curvature and use a Bochner type methodology to study how this Ricci curature effects the isometry group. 
 
 Here and in the sequel, we small use $\srm$ to denote the full curvature tensor associated to $\nabla$ and $R$ to denote the associated endomorphism of $TM$. Thus
 \begin{equation}
 \begin{split} R(A,B)C &= \left( \nabla_{A} \nabla_B -\nabla_B \nabla_A - \nabla_{[A,B]}\right)C,\\
 \srm(A,B,C,D) &= \aip{R(A,B)C}{D}{}.
 \end{split} 
 \end{equation}
 
  For sRC-manifolds with a large degree of symmetry, it should be expected that the curvatures are determined by the Lie algebra structure on $\mathscr{K}$ or $\mathscr{K}^*$. For our  results in this direction, we have the following pair of lemmas. The first which can be interpreted as stating that the curvature tensor measures the degree to which $\mathscr{K}^*_B$ fails to be a Lie subalgebra. 

\begin{lemma}\label{L:Bil}
For $K,L \in \mathscr{K}^*$, as operators on $HM$ we have
\[ \mB[{[K,L]}]  = [\mB , \mB[L] ] + \nabla_K \mB[L]  - \nabla_L \mB[K] - R(K,L). \]
\end{lemma}

\begin{proof} Let $X$ be a  horizontal vector field. Then
\begin{align*}
B_{[K,L]}(X) &= \left( \nabla_{[K,L]}X- [ [ K,L], X] \right)_H \\
&=  \nabla_{[K,L]}(X) - [ [ K,L], X] \\
& = \nabla_{[K,L]}(X) + [ [ L,X], K]  +[[X,K],L]\\
&= \nabla_{[K,L]}(X)  -\nabla_K [L,X] + \nabla_L [K,X] \\
& \qquad + \mB ([L,X])  - \mB[L] ([K,X])\\
&=  \nabla_{[K,L]}(X)  -\nabla_K \nabla_L X + \nabla_K ( \mB[L](X))   \\
& \qquad   -\nabla_L \nabla_K X + \nabla_L ( \mB[K](X))  + \mB ([L,X])  - \mB[L] ([K,X])\\
&= -R(K,L)X  + \nabla_K ( \mB[L](X))  + \nabla_L ( \mB[K](X)) \\
& \qquad   + \mB ([L,X])  - \mB[L] ([K,X]).
\end{align*}
To complete the argument, we note that
\begin{align*}
(\nabla_L \mB ) (X) &= \nabla_L (\mB(X)) - \mB (\nabla_L X) \\
\mB ([L,X]) &= \mB ( \nabla_L X + \mB[L] (X) ) 
\end{align*}
and the required identity follows easily.

\end{proof} 

The weak $H$-Killing fields in  $\mathscr{K}^*_B$ can be thought of as those that generate isometries that have no local rotation component, or are pure translation. From the previous lemma, we see that the commutation of such vector fields imposes a flatness condition on the manifold.

Next we see that if the manifold admits a lot of weak $H$-Killing fields then the curvature can be computed from properties of the operators $\mB$.

\begin{lemma}\label{L:Curv2}
For $K \in \mathscr{K}^*$ and horizontal vector fields $X,Y,Z $, 
\[ \begin{split} \srm (X,K,Y,Z) & = \aip{\nabla \mB (Y,X)}{Z}{} + \frac{1}{2} \mB( \tor(X,Y) ,Z) \\
 & \qquad -  \frac{1}{2}\mB( \tor(X,Z),Y) - \frac{1}{2}  \mB( \tor(Y,Z),X) . \end{split} \]
If $K \in \mathscr{K}$ then $R(X,K)Y = \nabla \mB(Y,X)$.
\end{lemma}

\begin{proof}  First we show that  
\begin{align*}
\aip{\nabla \mB  (Y,X)}{Y}{} & =\aip{\nabla_X \mB(Y) }{Y}{} - \mB(\nabla_X Y,Y) \\
 &= X \mB (Y,Y) - \mB(Y,\nabla_X Y)  - \mB (\nabla_X  Y ,Y)\\ 
 &=0.
\end{align*} 
Thus
\begin{equation}\label{E:skew1} \aip{\nabla \mB  (Y,X)}{Z}{}  = - \aip{\nabla \mB  (Z,X)}{Y}{}.\end{equation}
Now
\begin{align*}
&\aip{\nabla \mB  (Y,X)}{Z}{}  -\aip{\nabla \mB  (X,Y)}{Z}{}+  \aip{ \mB(\tor(X,Y))}{Z}{}\\
& \qquad  =  \aip{ \nabla_X \nabla_Y K_H -\nabla_Y \nabla_X K_H }{Z}{}+  \aip{ \mB(\tor(X,Y))}{Z}{} \\
& \qquad  \qquad + \aip{\nabla_X \tor(K,Y)-\nabla_Y \tor(K,X) }{Z}{}   -\aip{ \nabla_{\nabla_X Y -\nabla_Y X} K_H}{Z}{} \\ & \qquad  \qquad  -\aip{ \tor(K,\nabla_XY -\nabla_Y X)}{Z}{}   \\
&\qquad = R(X,Y)K_H - \aip{\nabla_{\tor(X,Y)} K_H}{Z}{}  +  \aip{ \mB(\tor(X,Y))}{Z}{}\\
& \qquad  \qquad + \aip{\nabla_X \tor(K,Y)-\nabla_Y \tor(K,X) }{Z}{}  \\ & \qquad \qquad  -\aip{ \tor(K, \tor(X,Y) +[X,Y] )}{Z}{}   \\
&\qquad = R(X,Y)K_H+ \aip{\nabla_X \tor(K,Y)-\nabla_Y \tor(K,X) }{Z}{}\\
& \qquad \qquad -\aip{\tor(K,[X,Y]) }{Z}{}  \\
&\qquad = R(X,Y)K_H+ \aip{ (\nabla \tor)(K,Y,X)  - (\nabla \tor)(K,X,Y) }{Z}{}  \\
&\qquad  \qquad  +  \aip{ \tor(\nabla_X K,Y)-\tor(\nabla_Y K,X)+ \Tor(K,X,Y) }{Z}{} \\
&\qquad = R(X,Y)K_H- \aip{  \mathcal{C} (\nabla \tor)(K,X,Y) }{Z}{}  \\
&\qquad  \qquad  +  \aip{ \tor(\tor(X,K),Y)-\tor(\tor(Y,K),X)+ \Tor(K,X,Y) }{Z}{} \\
&\qquad = \aip{ R(X,Y)K_H-  \mathcal{C} (\nabla \tor)(K,X,Y) +\mathcal{C}\Tor(K,X,Y) }{Z}{},
\end{align*}
where $\mathcal{C}$ represents the cyclic sum and $\Tor(A,B,C) = \tor(A,\tor(B,C))$. The Algebraic Bianchi Identity (see Lemma 3.4 in \cite{Hladky4}), implies that 
\begin{equation}\label{E:abi}
 \mathcal{C} R(A,B)C = \mathcal{C} (\nabla \tor)(A,B,C) - \mathcal{C} \Tor(A,B,C),
 \end{equation}
and so 
\begin{align*}
&\aip{\nabla \mB  (Y,X)}{Z}{}  -\aip{\nabla \mB  (X,Y)}{Z}{}+  \aip{ \mB(\tor(X,Y))}{Z}{}\\
& \qquad= \aip{ R(X,Y)K_H-  \mathcal{C} R(K,X)Y }{Z}{}\\
&\qquad = \aip{R(X,K)Y-R(Y,K)X }{Z}{}.\\
\end{align*}
Hence 
\begin{align*}
&\aip{ \nabla \mB(Y,X) }{Z}{} = - \aip{\nabla \mB  (Z,X)}{Y}{} \\
& \quad = -\aip{\nabla \mB  (X,Z)}{Y}{}- \aip{R(X,K)Z-R(Z,K)X + \mB \tor(X,Z) }{Y}{}\\
& \quad= \aip{\nabla \mB (Y,Z)}{X}{} - \aip{R(X,K)Z-R(Z,K)X }{Y}{} \\
& \quad= \aip{\nabla \mB (Z,Y)}{X}{}  - \aip{R(X,K)Z-R(Z,K)X + \mB \tor(X,Z)}{Y}{}  \\ & \qquad \quad +   \aip{ R(Z,K)Y-R(Y,K)Z  - \mB(\tor(Z,Y)}{X}{} \\ 
& \quad= -\aip{\nabla \mB(X,Y)}{Z}{}  - \aip{R(X,K)Z-R(Z,K)X + \mB \tor(X,Z) }{Y}{} \\
& \quad \qquad + \aip{R(Z,K)Y-R(Y,K)Z  - \mB \tor(Z,Y)}{X}{} \\
& \quad= -\aip{\nabla \mB(Y,X) }{Z}{} - \aip{R(X,K)Z-R(Z,K)X + \mB \tor(X,Z) }{Y}{}  \\
 & \quad\qquad +   \aip{R(Z,K)Y-R(Y,K)Z - \mB \tor(Z,Y) }{X}{}  \\
 & \quad \qquad -\aip{ R(Y,K)X-R(X,K)Y +\mB \tor(Y,X) }{Z}{}\\
 & \quad= -\aip{\nabla \mB(Y,X) }{Z}{} + 2\aip{R(X,K)Y}{Z}{} \\
 & \quad \qquad  +   \mB(\tor(X,Z),Y)+   \mB(\tor(Y,Z),X)+  \mB(\tor(Y,X),Z).
\end{align*}
The result immediately follows 

\end{proof}

A corollary to this last lemma will become useful later.

\begin{cor}\label{C:Ric1} For $K \in  \mathscr{K}$,
\[ \aip{ \text{tr }\nabla \mB }{Z}{} =  - \sum\limits_i  \srm(E_i,K,Z,E_i)  \]
where $E_i$ is any orthonormal frame for $HM$. 
\end{cor}



\subsection{Compact sRC-manifolds and Ricci Curvature}

In this sub-section, we study $H$-Killing fields on compact sRC-manifolds and look at the relation with a sRC analogue of the Ricci curvature. Throughout this section we shall assume that $M$ is a compact, oriented sRC-manifold equipped with the basic grading.

At first glance, the most natural generalization of the Ricci curvature to sRC-manifolds would appear to be
\[ \text{tr }\srm(A,B) = \sum\limits_k \srm(E_k,A,B,E_k) \]
where $E_k$ is any orthonormal frame for $HM$. However, this tensor is not in general symmetric  even when restricted to horizontal vectors.
\begin{defn}\label{D:sRicci}
The {\em sub-Ricci curvatures} of an sRC-manifold $M$ are the tensors
\begin{align*}
 \src(A,B) &= \text{tr }\srm(A,B) -  \frac{1}{2}  \sum\limits_k  \aip{\Tor(E_k,A_H,B_H)}{E_k}{}\\
 & \qquad -  \aip{\text{tr }\Tor(A_H)}{B_H}{} \\
 & \qquad -\sum\limits_k  \aip{ (\nabla \tor  - \Tor )(E_k,A_V,B_H) }{E_k}{} \\
& \qquad + \aip{\text{tr}(\nabla \tor - \Tor) (A_V)}{B_H}{}
\end{align*}
where $\{E_k\}$ is any horizontal orthonormal frame. 
\end{defn}

\begin{lemma}\label{L:asym}
The sub-Ricci curvature $\src$ is symmetric and satisfies
\[ \src (HM,VM)=\src(VM,HM) =\src(VM,VM) = 0 .\]
The sub-Ricci curvature is also independent of the choice of metric extension.
\end{lemma}

\begin{proof}
We first show symmetry on $HM$. Note that if $A,B$ are horizontal, then the last three terms vanish. Now it follows from the Algebraic Bianchi Identity \eqref{E:abi} that
\[ \mathcal{C} R(X,Y)Z = \mathcal{C} \nabla \tor(X,Y,Z)  -\mathcal{C} \Tor(X,Y,Z)\]
and hence from elementary properties of curvature that
\begin{align*} 2 \aip{ R(Z,X)Y}{W}{} & - 2\aip{R(W,Y)X}{Z}{} \\
& = \mathcal{C} \aip{ \mathcal{C} \nabla \tor(X,Y,Z)  -\mathcal{C} \Tor(X,Y,Z)}{W}{} 
\end{align*}
If $X,Y,Z,W$ are all horizontal, then the $\nabla \tor$ terms all vanish. Thus using elementary properties of $\Tor$, yields
\begin{align*}  \aip{ R(Z,X)Y}{W}{} & - \aip{R(W,Y)X}{Z}{} =\mathcal{C} \aip{\Tor(X,Y,Z)}{W}{}  
\end{align*}
If $Z=W$, then 
\begin{align*}  \aip{ R(Z,X)Y}{Z}{} & - \aip{R(Z,Y)X}{Z}{} = \aip{ \Tor(Z,X,Y)}{Z}{}\\
& \qquad + \aip{\Tor(Z,Z,X)}{Y}{}  - \aip{\Tor(Z,Z,Y)}{X}{}  
\end{align*}
Letting $Z$ run over an orthonormal frame for $HM$ thus produces the desired symmetry result.

The non-trivial part remaining is to show that $\src(T,X) =0$ for all $T \in VM$ and $X \in HM$, but  this follows from a similar argument again using the horizonal Bianchi identities. 

\end{proof}


If $M$ is $H$-normal then the sub-Ricci curvatures take on  a much more familiar form,

\begin{cor}\label{C:sym}
If $M$ is $H$-normal with $VM$ integrable, then 
\[ \src(A,B) = \sum\limits_k \srm(E_k,A_H,B_H,E_k)\]
and hence the latter is symmetric.
\end{cor}


One of our purpose in introducing the sub-Ricci curvatures is to use Bochner type results to study the relationship between curvature and symmetry on sub-Riemannian manifolds. To use this theory, we shall need a geometrically defined subelliptic Laplacian. 

\begin{defn}\label{D:HL}
For a tensor $\tau$, the {\em horizontal gradient} of $\tau$ is defined by
\[ \nabla_H \tau = \nabla_{E_i} \tau \otimes E_i, \]
the {\em horizontal Hessian} of $\tau$ by
\[ \nabla^2  \tau (B,A)=\left(  \nabla_{A} \nabla_{B} - \nabla_{\nabla_{A} B} \right)  \tau\]
for $X,Y \in HM$ and zero otherwise.
Finally, the {\em horizontal Laplacian} of $\tau$ is defined by
\[ \lp \tau = \text{tr }\left( \nabla^2  \tau \right)  =\left(  \nabla_{E_i} \nabla_{E_i} - \nabla_{\nabla_{E_i}E_i} \right)  \tau  \]
\end{defn}

The Laplacian on a Riemannian manifold has a rich and interesting $L^2$-theory. To replicate this for sRC-manifolds, it is necessary to choose a metric extension. This metric extension then yields a volume form and we have meaningful $L^2$-adjoints. Unfortunately, the horizontal Laplacian defined here, does not always behave as nicely as the Riemannian operator. However, if we make the mild assumption that $M$ has a vertically rigid metric extension, then it is shown in \cite{Hladky4} that $\lp$ is formally self-adjoint and on functions $\lp = - \nabla_H^* \nabla_H$.

The vertically rigid requirement comes from the observation (see \cite{Hladky4}) that for a vector field $A$,
\begin{equation} \text{div }A = \text{tr }\left( \nabla A \right)+  \mathfrak{R}(A)=  \text{tr }\left( \nabla A \right)+  \aip{\widehat{\mathfrak{R}}}{A}{}.\end{equation}
Without assuming vertical rigidity, many integration-by-parts results will include terms involving the rigidity vector that are hard to analyze. 

With this in place Corollary \ref{C:Ric1} can be re-interpreted and improved on for the $H$-normal case.

\begin{cor}\label{C:Ric} If  $M$ is $H$-normal and $VM$ is integrable, then for $K \in  \mathscr{K}$,
\[ \aip{ \lp K_H}{ Z}{} =  - \src(K,Z).  \]
\end{cor}

From this we can trivially obtain the following.

\begin{cor}\label{C:Par}
If $M$ is $H$-normal, $VM$ is integrable and $A$ is any parallel vector field, then for all $K \in \mathscr{K}$, the inner product $\aip{K_H}{A}{}$ is $H$-harmonic. 
\end{cor}

We can now begin the technical task of studying the relationship between the horizontal Laplacian, the sub-Ricci curvatures and $H$-Killing fields. The begin with the following technical lemma.
\begin{lemma}\label{L:KillingMix}
For $K \in \mathscr{K}$,
\[ \text{tr }\srm(K_V,K_H) = \aip{ \text{tr}(\nabla \tor-\Tor)(K_V) }{K_H}{}
 - K_H \aip{ \widehat{\mathfrak{R}}}{K_V}{}\]
\end{lemma}

\begin{proof} This is largely a straight-forward computation, but we shall make a preliminary remark first. Namely, due to a simple  symmetry/skew-symmetry argument we have
\begin{equation}\label{E:skew}
\sum\limits_k \aip{\tor(\nabla_{K_H} E_k,K_V)}{E_k}{}  = 0.
\end{equation}
Then using \eqref{E:skew} and Lemma \ref{L:EP} (c) we see
\begin{align*}
\sum\limits_k  & \aip{\nabla \tor(E_k,K_V,K_H)}{E_k}{} =\sum\limits_k \aip{ \nabla_{K_H} \tor(E_k,K_V)}{E_k}{}  \\
& \qquad \qquad - \aip{ \tor(\nabla_{K_H} E_k,K_V) + \tor(E_k,\nabla_{K_H} K_V)}{E_k}{}\\ 
&\qquad = K_H \aip{ K_V}{\widehat{\mathfrak{R}}}{}  - \sum\limits_k \aip{\tor(E_k, \tor(K_V,K_H)}{E_k}{} \\
&\qquad = K_H \aip{ K_V}{\widehat{\mathfrak{R}}}{} - \sum\limits_k \aip{\Tor(E_k,K_V,K_H)}{E_k}{}
\end{align*}
The result then follows easily from Lemma \ref{L:asym}.

\end{proof}

We next show how to integrate a key portion of the torsion.

\begin{lemma}\label{L:int2}
If $M$ is compact and vertically rigid with $VM$ integrable, then for $K \in \mathscr{K}$
\[ \int_M  \aip{ \text{tr}(\nabla \tor-\Tor)(K) }{K_H}{} dV = \int_M  \left| \tau_H^H(K) \right| ^2\; dV\]
\end{lemma}

\begin{proof}
Since $M$ is vertically rigid,
\begin{equation}\label{E:int1}
\begin{split}
\text{div} \tor(K,K_H)_H &= \sum\limits_k  \aip{\nabla_{E_k} \tor(K,K_H)}{E_k}{}  \\
& =\sum\limits_k   \aip{ \nabla \tor(K_V,K_H,E_k) + \tor(\nabla_{E_k} K_V,K_H)}{E_k}{} \\
& \qquad +\aip{\tor(K_V,\nabla_{E_k} K_H)}{E_k}{} \\
& = \sum\limits_k  \aip{  \nabla \tor(K,E_k,E_k) +\tor(\tor(E_k,K)_V,E_k) }{K_H}{} \\
& \qquad  + \aip{\tor(K,E_k)}{\mB(E_k)}{} - \left| \tor(K,E_k)_H \right|^2\\
& = \aip{ \text{tr}(\nabla \tor-\Tor)(K) }{K_H}{}- \left| \tau_H^H(K) \right| ^2
\end{split}
\end{equation}
and the result follows immediately.

\end{proof}

\begin{cor}\label{C:int2} If $M$ is compact and totally rigid with $VM$ integrable, then for $K \in \mathscr{K}$
\[ \int_M \text{tr }\srm  (K,K_H) dV  = \int_M \src (K_H,K_H) + \left| \tau_H^H(K) \right| ^2 \; dV \]
\end{cor}

\begin{proof} First note that since we are assuming total rigidity then
\begin{align*}
\text{tr }\srm(K_V,K_H)& = \aip{ \text{tr}(\nabla \tor-\Tor)(K_V) }{K_H}{} \\
&= \aip{\text{tr}(\nabla \tor -\Tor)(K)}{K_H}{} - \aip{ \text{tr }\Tor(K_H) }{K_H}{}
\end{align*}
Thus \[ \begin{split} \text{tr }\srm(K,K_H) &= \src (K_H,K_H)+ \aip{\text{tr}(\nabla \tor -\Tor)(K)}{K_H}{}  \\
& \qquad - K_H \aip{ \widehat{\mathfrak{R}}}{K_V}{} \end{split}\]

\end{proof}

We are now finally in a position to state our main Bochner result for $H$-Killing fields.

\begin{thm}\label{T:gradalt}
If $K \in \mathscr{K}$ and $Z \in HM$, then
\begin{enumerate}
\item $\displaystyle \text{tr }\srm (K,Z) = - \aip{ \lp K_H +\text{tr}( \nabla \tor-\Tor)(K) }{Z}{}$
\end{enumerate}
If $M$ is totally rigid with $VM$ integrable, then furthermore
\begin{enumerate}
\addtocounter{enumi}{1}
\item $\displaystyle  \int_M \aip{\lp K_H}{K_H}{} dV = \int_M -\src (K_H,K_H)  - 2  \left| \tau_H^H(K) \right| ^2$
\item $\displaystyle 0 = \int_M \frac{1}{2} \lp \left| K_H \right|^2 dV = \int_M -\src (K_H,K_H)  -   \left| \tau_H^H(K) \right| ^2 + \left|\mB  \right|^2 \; dV$
\end{enumerate}
\end{thm}

\begin{proof}
From Lemma \ref{L:Curv2}, we see that
\begin{align*}
 -\text{tr }\srm (K,Z) &=\aip{ \text{tr }\nabla \mB}{Z}{}  \\
 &= \sum\limits_i   \aip{ \nabla_{E_i}  \left( \nabla_{E_i} K_H + \tor(K,E_i) \right) }{Z}{}  \\& \qquad -\sum\limits_i \aip{\nabla_{\nabla_{E_i} E_i}  K_H - \tor(K,\nabla_{E_i} E_i) }{Z}{}\\
 &=  \aip{ \lp K_H +\text{tr }(\nabla \tor-\Tor)(K) }{Z}{}.
\end{align*}
The remaining results then follow easily from Lemma \ref{L:KillingMix},  Corollary \ref{C:int2} and the observation that
\[ \left| \mB \right|^2 = \left| \nabla_H K_H \right|^2 + \left| \tau_H^H(K) \right|^2 .\]
\end{proof}
This result can be interpreted as saying that the Ricci curvature measures the difference in size between the symmetric and skew-symmetric parts of $\nabla K_H$. The most useful results come in the $H$-normal case, where $\nabla K_H$ is known to be skew-symmetric

This Bochner type theorem can be used to derive some consequences of curvature on the space of Killing forms. For negative curvatures, we have the following generalization of a classical result of Bochner.

\begin{lemma}\label{L:Bochner2}
Suppose that $M$ is $H$-normal, $VM$ is integrable and $\src(X,X) \leq 0$ for all $X \in HM$, then for all $K \in \mathscr{K}$, $\nabla K_H \equiv 0$. Thus $\mathscr{K}=\mathscr{K}_B$.

Furthermore if there is some $p \in M$  such that $\src(X,X)<0$ for all $X \in H_pM -\{0\}$  then $\mathscr{K} = \mathscr{K}_V$.\end{lemma}

\begin{proof} For the main part of the corollary, we simply note that if $M$ is $H$-normal and $VM$ integrable then Theorem \ref{T:gradalt} implies 
\begin{equation}\label{E:IntL} 0 = \int_M \lp f = \int  \left( \left| \nabla_H K_H \right|^2 - \src(K_H,K_H) \right) \geq \int \left| \nabla_H K_H \right|^2 \geq 0  .\end{equation}
This clearly implies that $\nabla_H K_H=0$. However if $T$ is a section of $VM$, then
\[  \nabla_K T - \nabla_T K - [K,T] = \tor(K,T) \in VM.\]
But then projecting to $HM$, we see $\nabla_T K_H=0$, so $K_H$ is parallel.

Now if there is a point $p$ where $\src$ is strictly negative, then we would  have an impossible strict inequality in \eqref{E:IntL} unless $K_H = 0$ at $p$. As $K_H$ is parallel, this then implies that $K_H \equiv 0$.

\end{proof}

The Heisenberg group in  Example \ref{X:Heisn} shows that the compactness condition is necessary here. From this we can obtain a sequence of easy corollaries.

\begin{cor}\label{C:negnot}
There are no compact, $H$-normal, sRC-manifolds with $VM$ integrable and quasi-negative Ricci curvature that are homogeneous under the action of $\text{Iso}(M)$
\end{cor}

\begin{proof}
Under these conditions $\mathscr{K} =\mathscr{K}_V$ and hence the dimension of the group $\text{Iso}(M)$ is at most $\dim VM$. 
\end{proof}

\begin{cor}\label{C:flat2}
If $M$ is $H$-normal, $\src(X,X) = 0$ and $\dim \mathscr{K} = \dim M$ then $M$ is parallelizeable and $H$-flat.
\end{cor}

\begin{proof}
Since $\dim \mathscr{K} = \dim M$ and $\mathscr{K}=\mathscr{K}_B$, we see that there is a global frame of $H$-Killing fields.

From Lemma \ref{L:Bil} it follows immediately that $M$ is $H$-flat.

\end{proof}

Positive Ricci curvature appears to place fewer restrictions on the Killing fields. However for manifolds with positive sectional curvature, we have the following generalization of a classical theorem of Berger  (see \cite{Berger} or \cite{Wallach} ).

\begin{lemma}\label{L:Berger}
Suppose $M$ is $H$-normal   with $VM$ integrable and  $\text{dim}(HM)$ even. If every horizontal sectional curvature is positive, then every Killing field $K$ is vertical at some point $p\in M$.
\end{lemma}

\begin{proof}
Choose $K \in \mathscr{K}$ and set $f = \frac{1}{2} \left| K_H \right|^2$. Then Lemma \ref{L:Curv2}  implies that for any horizontal vector field $E$,
\begin{align*}
 \nabla^2 f (E,E) &= \aip{\nabla^2 K_H(E,E)}{K_H}{} + \left| \nabla_E K_H \right|^2\\
 &= \aip{ R(E,K)E}{K_H}{} + \left| \nabla_E K_H \right|^2\\
 &= -R(E,K_H,K_H,E) + \left| \mB(E) \right|^2
 \end{align*}
 where the last line follows from Lemma \ref{L:KillingMix} and $H$-normality.

Recall that for all $p \in m$, the map $\left( \mB \right)_{|p} \colon H_pM \to H_pM$ is skew-symmetric. Now if $v  \in \ker{(\mB)_{|p}} $ we see that
\[\nabla^2 f (v,v)  = -R(v,K_H,K_H,v.)\]
If $v$ is linearly independent from $K_H$, then the positive sectional curvature constraint implies that $\nabla^2 f(v,v)<0$. 

Now let $p\in M$ be the point at which $f$ attains it's minimum. Then at $p$,  $\nabla f =0$ and $\nabla^2 f \geq 0$. But at $p$, we also have $K_H \in \ker{(\mB)_{|p}} $. As $HM$ is even dimensional, $\ker{(\mB)_{|p}} $ is also even dimensional and so if $K_H \ne 0$ then $\ker{(\mB)_{|p}} $ must contain a vector $v$ linearly independent from $K_H$. This is a contradiction and $K_H$ must vanish at $p$.

\end{proof}

For the final result in this section, we illustrate one way in which several results of this paper can be combined. We note that $M$ has positive sectional curvatures and a full set of purely vertical $H$-Killing fields then a large number of  components of the curvature must vanish identically.

\begin{cor}\label{C:Berger2}
Under the same conditions as Lemma \ref{L:Berger}, if $\dim \mathscr{K}_V = \dim VM$ then  $\mathscr{K}_B = \mathscr{K}_V$ and  hence
\[ R(T,U)X =0 \]
for all $T,U \in VM$ and $X \in HM$.
\end{cor}

\begin{proof}
As $M$ is $H$-normal, from the basic definitions, we must have $\mathscr{K}_V \subseteq \mathscr{K}_B$. But as $\dim \mathscr{K}_V = \dim VM$,  at every point $p \in M$, the first component of the map $\chi_p$ from Theorem \ref{T:Basic} must induce a  bijection between $\mathscr{K}_V$ and $V_pM$. 

Suppose $K \in \mathscr{K}_B$. As  every $H$-Killing field must be purely vertical at some point of $M$, there must be a point $p$ and $L \in \mathscr{K}_V$ such that $\chi_p (K) = \chi_p(L) \in V_p M \times \{0\} $. But as $\chi_p$ is injective by Theorem \ref{T:Basic}, this implies that $K = L \in \mathscr{K}_V$.

Since $VM$ is integrable, we see that $\mathscr{K}_V$,  and hence  $\mathscr{K}_B$, is a Lie subalgebra of $\mathscr{K}$.  The result then follows easily from  Lemma \ref{L:Bil}.

\end{proof}


\section{Examples}\label{Ex}

\subsection{Carnot Groups}

A Carnot group (of step size $r$) is a connected, simply connected Lie group $G$ with a stratified Lie algebra $\mathfrak{g}$ 
\[ \mathfrak{g} = \mathfrak{g}_1 \oplus \dots  \oplus \mathfrak{g}_r \quad \text{such that} \quad [\mathfrak{g}_i ,\mathfrak{g}_j] = \mathfrak{g}_{i+j} . \]
The horizontal bundle $HG$ is spanned by the left translates of $\mathfrak{g}_1$ with $VM$ spanned by the left translates of $\mathfrak{g}_2 \oplus \dots \oplus \mathfrak{g}_r$. To complete the sRC-structure, we assume that $G$ is equipped with a left invariant metric on $HG$.

It is easy to see that this sRC-structure for $G$ is $H$-normal. It is typical in the literature to extend the metric to a left-invariant metric that respects the stratification. However when $r>2$ this leads to a vertically rigid, but not $V$-normal structure. We shall take a different approach here.  

The left translations of $G$ clearly are sRC-isomorphisms. Thus the one- parameter subgroups of $G$ induce a subspace of $\mathscr{K}$ with dimension $\dim G$. Now since the Lie bracket of any left invariant vector fields has no horizontal  component,  it is easy to check that any left-invariant horizontal vector field is parallel. This implies that if $K$ is an $H$-Killing field associated to a one-parameter subgroup, then for any left-invariant horizontal vector field $X$
\[ 0=[ K,X] = \nabla_K X - \nabla_X K_H - \tor(K,X)_H =  - \nabla_X K_H.\]
Thus $\mB \equiv 0$. This implies that $K \in \mathscr{K}_B$ and so $\dim \mathscr{K}_B = \dim G$. But this immediately implies that $\dim \mathscr{K}_V = \dim VG$. Hence by Corollary \ref{C:HNorm}, $G$ admits a strictly normal metric extension. We shall call such a metric a Killing metric for $G$. In general however, a Killing metric may not be left-invariant. 

 If we let $X_1,\dots,X_k$ be an orthonormal left-invariant frame for $HG$ and set $\theta^1,\dots,\theta^k$ be the dual frame of $1$-forms (which also annihilate $VG$), then it is clear that each $\theta^i$ is closed. Since any Carnot group is known to be diffeomorphic to Euclidean space, this implies that each $\theta^i$ is globally exact and hence there are functions $x^1,\dots,x^k$ such that $X_i x^j = \delta_i^j$. We shall call these the horizontal coordinates for the frame  $X_1,\dots,X_k$.

Since all left-invariant horizontal vector fields are parallel, from Lemma \ref{L:EP} (f) and Lemma \ref{L:Curv2} we see that  $\nabla K_H(VG) \equiv 0$ and $\nabla^2 K_H \equiv 0$. Thus the coefficients of $K_H$ with respect to any left-invariant orthonormal frame must be affine linear functions of the  corresponding horizontal coordinates.

\begin{exm}\label{X:Heisn2}
We consider the Heisenberg groups of Example \ref{X:Heisn}. If we ordering the global  basis of left invariant vector fields by $X_1,\dots,X_n,Y_1,\dots,Y_n$, the torsion matrix is
\[ \breve{T} = \begin{pmatrix} 0 & I_n \\ -I_n & 0 \end{pmatrix} \otimes T .\]
As $\dim V\mathbb{H}^{n} =1$,  we immediately see from the examples of Example \ref{X:Heisn} and Lemma \ref{L:vd2}  that all isotropy groups $\text{Iso}_p(\mathbb{H}^{n})$ are isomorphic to $U(n)$. Hence as vector spaces
\[ \mathscr{K}^*(\mathbb{H}^{n})  \cong \text{Lie}(\mathbb{H}^{n}) \oplus  \mathfrak{u}(n).\]

\end{exm}

\begin{exm}\label{X:Engel}
The Engel group $G$ is $\mathbb{R}^{4}$ with  $HG =\text{span} \langle X,Y \rangle$ and $VM =\text{span} \langle T_1,T_2 \rangle$ are spanned by the left-invariant vector fields
\[ X = \pd{}{x} -y \pd{}{t^1} - t^1\pd{}{t^2} , \quad  Y =\pd{}{y} , \quad T_1= \pd{}{t^1}  , \quad T_2 = \pd{}{t^2}\]
The left-invariant metric on $HM$ is defined by declaring $X,Y$ to be orthonormal. The space of vertical $H$-Killing fields is easily seen to be spanned by
\[ S_1= T_1-x T_2 , \qquad S_2 =T_2.\]
If we let $V^1 = \text{span}(T_1)$ and $V^2 =\text{span}(T_2)$ then we have a regular grading for $M$. It is also easy to verify that
\begin{equation}\label{E:englie} \mathscr{K}^R_B = \text{span} \left\langle S_1, \; S_2, \;   Y -xS_1\;  X-yS_1 -(xy+t^1)S_2 \right\rangle .\end{equation}

In the notation of Theorem \ref{T:higher}, applied with $m=1$,  we see that $L =\langle Y \rangle$.  But then $\mB$ is a skew-symmetric linear map on a one dimensional vector space. Hence $\mB \cong 0$ at $p$. Thus $\mathscr{K}^R \cong  \mathscr{K}^R_B$.

In fact, in this example we can use flatness to get a stronger result.  If we work with the basic grading, a Killing metric can be defined by declaring $S_1,S_2$ to be an orthonormal frame for $VM$.  Let $p=(0,0,0,0)$ and let $K \in \mathscr{K}_p$. Then using the general results from Carnot groups, we see that up to a constant rescaling we must have
\[ K_H = y X - x Y\]
But then
\begin{align*}
 0 &= \tor([K,X],Y) + \tor(X,[K,Y]) = [K,\tor(X,Y)] \\
 &= [ K, S_1 +xS_2]  = [K,S_1]  + y S_2 + x [K,S_2] 
\end{align*}
But $[K,S_i]$ must be a purely vertical $H$-Killing field for $i=1,2$ and hence must be a constant linear combination of $S_1,S_2$. This is a contradiction and hence $K \equiv 0$. Thus by dimension count,
 for the Engel group 
 \[ \mathscr{K} \equiv \mathscr{K}_B \equiv \mathscr{K}^R_B =\mathscr{K}^R .\]
\end{exm}

\subsection{Strictly pseudoconvex pseudohermitian manifolds}

A pseudohermitian manifold $M^{2n+1}$ is a odd dimensional manifold equipped with a non-vanishing $1$-form $\eta$ and an endomorphism $J \colon \ker{\eta} \to \ker{\eta}$ with $J^2=-1$. The manifold is strictly pseudoconvex if the bilinear form $d\eta(X,JY)$ is positive definite on $\ker{\eta}$ and hence is a sub-Riemannian metric for $\ker{\eta}$. There is then a unique characteristic vector field $T$ such that $\eta(T)=0$, $d\eta(T,\cdot)=0$.  Setting
\[ HM = \ker{\eta}, \qquad VM = \langle T \rangle \]
makes $M$ an sRC-manifold and we can the define the Levi metric extension by  defining $JT=0$ and setting
\[ g(A,B) = d\eta(A,JB) + \eta(A)\eta(B).\]
It was shown in \cite{Hladky4} that the sRC-connection associated to the Levi metric is exactly the well-known Tanaka-Webster connection.  It is then easy to see that with this metric extension $M$ is totally rigid and $V$-normal.

If we work with a $J$ graded orthonormal  frame $X_1,\dots,X_n,JX_1,\dots,JX_n$ the torsion operator $\tor \colon H_p M \times H_p M \to VM$ can be identified with a skew-symmetric matrix  
\begin{equation}\label{breve}
\breve{T} = \begin{pmatrix} 0 & I_n \\ -I_n & 0 \end{pmatrix}  \otimes T.
\end{equation}

Since $\dim VM=1$, we see that $M$ is vertically discrete.  Indeed $\Omega_p = [-1,1] \otimes T$ so $GL(\Omega_p)=\{ \pm 1\}$  It is then clear from  remarks following Lemma \ref{L:vd2} that $\text{Iso}_p(M) \subseteq U(n)$. From this we immediately have that for all $K \in \mathscr{K}$,  $[K,JX]=J[K,X]$. Since $ \nabla JX = J \nabla X$, this immediately implies that for all $X \in HM$
\[ \mB (JX) = J \mB(X).\]  
Thus $\mathscr{K}$ consists of the Riemannian Killing fields for the Levi metric that preserve the decomposition $TM=HM \oplus VM$.

For $p \in M$, the operator $\mathcal{T}^0_T (X) = \tor(T,X)$ is self-adjoint with respect to the pointwise inner product and so has a basis of eigenvectors. Since $\mathcal{T}^0_T$ anti-commutes with $J$, the eigenvalues come in pairs $\pm \lambda$ with $J \colon E_{\lambda} \to E_{-\lambda}$ an isomorphism.  Thus there is an orthogonal decomposition
\[ H_pM = E_0  \oplus \left( E_{\lambda_1} \oplus E_{(-\lambda_1)} \right)\oplus \dots \oplus  \left(E_{\lambda_k} \oplus E_{(-\lambda_k)} \right) \]
Let $2h_0 = \dim(E_0)$ and $h_i = \dim(E_{\lambda_i})=\dim(E_{-\lambda_i})$ for $i>0$.

 By \eqref{E:commutes}, if $K_p=0$ then $\mB$ and $\mathcal{T}^0_T$ commute at $p$. Thus at $p$, $\mB$ decomposes into skew-symmetric operators $E_{\lambda_i} \to E_{\lambda_i}$ on each of the eigenspaces of $\mathcal{T}^0_T$.

For $\lambda_i>0$, the dimension of skew-symmetric bilinear forms on $E_{\lambda_i}$ is given by $\frac{h_i(h_i-1)}{2}$. Furthermore, Theorem \ref{T:tbound2} and  \eqref{breve} imply that the component of $\mB$ on $E_{-\lambda_i}$ is completely determined by the component  on $E_{\lambda_i}$.

Since $E_0$ is $J$-invariant, we see that $\left( \mB \right) _{|E_0} \in U(h_0)$ which has dimension  $h_0^2$. Putting these together implies that $\dim \mathscr{K}_p  = h_0^2 +  \sum\limits_i \frac{h_i(h_i-1)}{2}$ and hence
\[ \dim \mathscr{K} \leq 2n+1+h_0^2 +  \sum\limits_{i} \frac{h_i(h_i-1)}{2}.\]

\begin{cor}\label{C:psidim}
Suppose $M$ is an $H$-normal strictly pseudoconvex pseudohermitian manifold of dimension $2n+1$. Then
\[\dim \mathscr{K} \leq (n+1)^2 \]
\end{cor}

\begin{proof} Since $H_pM = E_0$ for all $p \in M$, we can take $h_0 =n$.

\end{proof}


\subsection{The Lie group $SO(n)$.}

Let $G$ be the Lie group $SO(n)$ with Lie algebra identified with the space of skew-symmetric $n\times n$ matrices. For notational purposes let $E_{ij}$ represent the matrix with $+1$ at position $i,j$ and zeros everywhere else.

Now for $i>2$, let $X_{i}$ be the left-invariant vector field extending $E_{1i}-E_{i1}$ and declare $X_2,\dots,X_n$ to be an orthonormal basis for $HG$, the subbundle of $TG$ that they span.
Next for $1<i<j$ let $T_{ij}$ be the left invariant vector field extending $E_{ij}-E_{ji}$ and let $VG$ be the bundle spanned by all such $T_{ij}$. This defines a sRC-manfiold structure for $G$. Declaring $\{T_{ij} \colon 1<i<j \}$ to be an orthonormal basis for $VG$ defines a bi-invariant metric extension.

The bracket structure of $G$ is then given by
\begin{align*}
 [X_i, X_j] =-T_{ij},& \qquad i<j\\
 [X_i, T_{ij}] =  X_j,& \qquad [X_i, T_{ji} ] = -X_j\\
  [T_{ij},T_{ik}] = -T_{jk},& \qquad [T_{ij},T_{ki}] =T_{jk}
 \end{align*} 
 with all other brackets zero. It is then easy to check that with this structure $G$ is $H$-normal and that the metric extension is strictly normal. We can also easily compute that $\dim HG = n-1$, $\dim VG = \frac{(n-1)(n-2)}{2}$ and $\dim G = \frac{n(n-1)}{2}$. From Corollary \ref{C:dim}, we have the estimate
\begin{equation}\label{E:Oest} \begin{split} \dim \text{Iso}(G) &\leq \dim G + \dim \text{Iso}_e (M) \\
& \leq  \frac{n(n-1)}{2}  + \frac{(n-1)(n-2)}{2} = (n-1)^2 \end{split}\end{equation}
where $e \in G$ is the identity. Indeed, it is clear that left multiplication provides a $\frac{1}{2} n(n-1)$ dimensional family of $H$-isometries with no fixed points. Furthermore the left invariant vector fields $T_{ij}$  are themselves $H$-Killing fields. So the only question is whether the dimension $\text{Iso}_e(G)$ is actually equal to  $\frac{1}{2} (n-1)(n-2)$ . 

With respect to the obvious orthonormal frame for $HG$, the torsion matrix at any point is $\breve{T}$, the skew-symmetric $(n-1) \times (n-1)$ matrix with $\breve{T}_{ij} =T_{(i-1)(j-1)}$ for $1 \leq i<j \leq n$.
Torsion viewed as a map $H_p G \times H_p G \to V_pG$ thus has essentially the same structure as the wedge-product viewed as a map $H_pG  \times H_pG \to \Lambda^2 H_p G$. From this it is easy to see that $\Omega_p$ is the intersection of the unit ball in $V_pG$ with the space of decomposable elements. The symmetry group of this set is large and so we cannot use torsion to reduce the dimension of $\text{Iso}(G)$.

Now for $A \in T_eG $, we can consider the $1$-parameter subgroup $e^{tA}$ and define the conjugation subgroups of diffeomorphisms $\Phi_{A,t} \colon G \to G$ by $\Phi_{A,t} (g) =e^{tA} g e^{-tA}$. Using a Taylor expansion, we see that for  $y = e^{sB} \in G$ 
\[\Phi_{A,t} y= y + t[A,y] + \frac{t^2}{2} [A,[A,y]] + \frac{t^3}{3!} [A,[A,[A,y]]]+\dots\] 
and so at $e$
\begin{equation}\label{E:phi} (\Phi_{A,t})_* B =B+  t[A,B] + \frac{t^2}{2} [A,[A,B]] + \frac{t^3}{3!} [A,[A,[A,B]]]+\dots.\end{equation}
Since  any elements $A \in V_eG$ generates an $H$-Killing fields,  it is clear that for $A \in V_eG$ the pushforward  $(\Phi_{A,t})_*$ maps $H_eG \to H_eG$ and $V_eG \to V_eG$.  Now we can apply the observation that $\Phi_{A,t} \circ L_g  = L_{\Phi_{A,t}(g)} \circ \Phi_{A,t}$ to see that the maps $\Phi_{A,t}$ with $A \in V_eG$ preserve the splitting everywhere. As the standard metric is bi-invariant it now follows that for any $A \in V_eG$ and $t\in \mathbb{R}^{}$ that $\Phi_{A,t} \in \text{Iso}_e(G)$.

Now if $K_A \in \mathscr{K}_e$ is the corresponding $H$-Killing field corresponding to the subgroup $\Phi_{A,t}$, then differentiating \eqref{E:phi} appropriately, we see that for any left-invariant vector field $Y$
\[ [K_A,Y]_e = - [ A,Y_e] \]
where the bracket on the right is the Lie algebra bracket in $\mathfrak{o}(n)$. These are distinct for distinct elements $A \in V_eG$ and so we do indeed have $\dim \mathscr{K}_e =\frac{(n-1)(n-2)}{2}$ and so \eqref{E:Oest} is the optimal estimate. 

One consequence of this is that the groups $SO(n)$ fill a role in complemented sub-Riemannian geometry that they do not in standard Riemannian geometry. Namely, they are the model spaces of step $2$, compact, homogeneous sRC- manifolds with maximal symmetry groups and maximal vertical dimension.

\subsection{The Lie group $SL_n(\mathbb{R}^{})$.}

Let $G = SL_n(\mathbb{R}^{})$, the Lie group of $n \times n$ matrices with determinant $+1$.  The Lie algebra $\mathfrak{g}$ can then be identified with trace free $n \times n$ matrices. For $i \ne j$, let $X_{ij}$ be the left invariant vector field generated by $E_{ij}$ and for $i=1,\dots,n-1$ set $T_i$ to be the left invariant vector fields generated by $E_{ii}-E_{(i+1)(i+1)}$. Define an sRC-structure on $G$ by letting $X_{ij}$, $i\ne j$ be a global orthonormal frame for $HG$ and let $T_1,\dots,T_{n-1}$ be a global frame for $VG$.   Then $\dim HG = n(n-1)$ and $\dim VG = n-1$. 

Left multiplication then provides a transitive family of sRC-isometries and so $G$ is homogeneous. Then Corollary \ref{C:dim}, estimates
\[ \dim \text{Iso}_e (G) \leq  \frac{1}{2} n(n-1)( n^2-n-1).\] 
In this instance however, we shall be able to greatly reduce this estimate using torsion. 

The only Lie brackets of horizontal left invariant vector fields that produce vertical terms are (with $i<j$)
\[ [X_{ij} ,X_{ji}] = E_{ii}-E_{jj} = \sum\limits_{k=i}^{j-1} T_k .\]
Therefore there is a weak torsion decomposition of $H_eG$ into two dimensional blocks. By Corollary \ref{C:2dsplit}, the Lie group $\text{St}(\Omega_p^1)$ is therefore discrete. Therefore for any $U \in V^1_e$ and $K \in \mathscr{K}_e$ we must have $[\mB , \mathcal{T}_U^0]=0$. 
Now 
\[ [T_i, X_{jk}]= (\delta_{ij} - \delta_{(i+1)j} -\delta_{ik} + \delta_{(i+1)k}) X_{jk}.\]
An immediate consequence of this is that $\mathcal{T}_U^0$ is diagonal for any $U \in V_e^1$ with respect to orthonormal basis above. This means that all these operators commute and hence $\mB$ must preserve the intersections of the eigenspaces of all operators $\mathcal{T}_U^0$. It is straightforward to verify that each $X_{ij}$ generates an intersection of eigenspaces and so $\mB$ itself must be diagonal. As $\mB$ must also be skew-adjoint, we see that $\mB =0$ for all $K \in \mathscr{K}$. Hence $\text{Iso}(G)$ is isomorphic to a finite number of disjoint copies of  $G$ itself.

\subsection{The rototranslation group.}

A complemented sub-Riemannian manifold that arises from problems in neurobiology and computer imaging is the rototranslation group, (see for \cite{CS} and \cite{HP} ). This is the group $G = \mathbb{R}^{2} \times \mathbb{S}^{1}$ with group structure
\[ (a,b,\gamma) \cdot (x,y, \theta)  = (a+ x \cos \gamma - y \sin \gamma, b+ y\cos \gamma+ x\sin \gamma, \gamma + \theta).\]
The identity is $e =(0,0,0)$ and
\[ (x,y,\theta)^{-1} = ( -x\cos \theta- y\sin \theta,x \sin \theta-y \cos \theta,-\theta) .\]
An sRC-structure is created by setting $X,\Theta$ to be an orthonormal frame for $HG$ and $T$ a global frame for $VG$ where
\begin{align*}
X &= \cos \theta \pd{}{x} + \sin \theta \pd{}{y}, \quad 
\Theta  = \pd{}{\theta}, \quad T  = \sin \theta \pd{}{x} - \cos \theta \pd{}{y}.
\end{align*}
It is easy to check that this frame and the metric on $HG$ is left-invariant. The bracket structure on $G$ is then
\[ [X,\Theta]=T, \qquad [X,T]=0, \qquad [\Theta,T] =X.\]
From this we can compute the connection and torsion as
\begin{align*}
\nabla_X \Theta &=0, \qquad \nabla_\Theta X = 0, \qquad \nabla T =0,\\
\nabla_T X &= \Theta/2, \qquad \nabla_T \Theta = -X/2\\
\tor(T,\Theta) &= X/2, \qquad \tor(T,X)= \Theta/2, \qquad \tor(X,\Theta) = -T.
\end{align*}
As $VG$ is 1-dimensional, the rototranslation group is vertically discrete at step $2$ and so the operator $\mathcal{T} (\cdot) = \tor(T,\cdot)$ commutes with $\mB$ on $H_pG$ for any $K \in \mathscr{K}_p$. Now $\mathcal{T}$ has eigenvalues $\pm 1/2$ and so splits $H_pG$ into two 1-dimensional eigenspaces. Hence $\mB =0$ at $p$. 

It can then be checked by direction computation that there is a  global frame of non-vanishing Killing fields corresponding to the right-invariant vector fields.
\begin{align*}
\widehat{X}  & = \pd{}{x}  = \cos \theta X +\sin \theta T, \quad  \widehat{Y}  = \pd{}{y} = \sin \theta X -\cos \theta T\\
\widehat{\Theta}& = \pd{}{\theta} -y \pd{}{x} + x \pd{}{y} = \Theta + (x \sin \theta -y \cos \theta)X -(y \sin \theta + x \cos \theta)T
\end{align*}
Thus $\mathscr{K}$ is $3$-dimensional and spanned by the list above. 

In fact, similar to the Heisenberg groups, we can make a stronger statement. Suppose $K =a X + b \Theta + c T$ is a weak $H$-Killing field then
\begin{align*}
 b+Xc&= 0, \qquad -a+\Theta c=0 \\
 Xa=\Theta b&=0, \qquad Xb+\Theta a+c =0.
\end{align*}
Then 
\[ Tc = X\Theta c - \Theta Xc = Xa + \Theta b  =0\]
Now
\[ Tb = -TX c = -XT c =0, \qquad Ta = T \Theta c = \Theta T c -Xc =  b. \]
Since
\[ [T,K] = (Ta -b)X + Tb \; \Theta + Tc \; T \]
we immediately see that $K$ is a strong $H$-Killing field. Since $G$ admits a transitive family of non-vanishing $H$-Killing fields, we must therefore have  $\mathscr{K}^* = \mathscr{K}$.

\bibliographystyle{plain}
\bibliography{References}
\end{document}